\documentclass[fleqn]{article}
\usepackage{graphicx} 
\usepackage{multirow}
\usepackage{cite}
\usepackage{xcolor}
\usepackage{bm}
\usepackage{caption}
\usepackage{float}
\usepackage{amsmath,amsfonts,amssymb,amsthm,amsrefs}
\usepackage [english]{babel}
\usepackage [autostyle, english = american]{csquotes}
\usepackage[linguistics]{forest}
\usepackage{pdfpages}
\usepackage{url}
\usepackage{nicematrix}
\allowdisplaybreaks

\title{Complementary families of approximating polynomials with applications to finite element methods applied to differential equations of arbitrary even spatial order}
\author{Peter Moore}

\begin{document}

\maketitle

\begin{abstract}
Complementary families of polynomials are introduced to generate $C^m$ finite element basis functions of order $p \geq 2m+2$ for
arbitrary $m \ge 0$.  One family consists of the Hermite splines that serve as the nodal basis functions by ensuring $C^m$ continuity across element boundaries.  Explicit formulas for these splines for any $m \ge 0$ are presented on the canonical interval $[0,1]$.  The second family is derived on the interval $[-1,1]$ from derivatives of order $m+1$ of the Legendre polynomials of degree $p-m-1$ multiplied by binomial powers of degree $m+1$ at -1 and 1, respectively.  These polynomials, related to the ultraspherical polynomials, serve as the interior or bubble basis functions.  A relationship between the two families of polynomials is demonstrated.  For a particular $m$ and $p$, an interpolant is constructed using these basis pairs together with the roots of the related ultraspherical polynomial and the interval endpoints.  A formula for the interpolation error that extends the results for $m=0$ and $m=1$ is given.  To prove the formula extensions of the Lagrange interpolants are introduced.  A superconvergence result along with the related asymptotic equivalence of the interpolant and finite element solution is proved in the linear case in $H^{m+1}$. Computational results demonstrate the theory for a model problem. 
\end{abstract}

\section{Introduction}

Finite element, finite difference, and spectral methods have long been employed in solving partial differential equations in one, two, and three space dimensions. Less has been done as the order of the equations increases.  One exception to this is the application of Hermite splines of any order to solving two-point boundary problems of arbitrary even order \cite{BSV1968}.  The authors provide {\it a priori} error bounds on the finite element error but no computational results.

A dearth of models governed by higher-order equations may also explain the lack of interest.  There are several exceptions to this.  One involves blow-up solutions in one dimension for semilinear parabolic equations of order {\textit 2m+2} \cite{BGW2004,GMP2012,GNZ2020}.  A second example comes from extensions of the seminal paper by Mindlin of his theory of deformation of elastic solids \cite{M1965}.  Both sixth- \cite{KN2018,LGZ2017} and eighth- \cite{BMM2019} order equations have been derived from it to model the deformation of solids.  A third example, for which fourth- and sixth-order equations have been studied, is a thin film model; see \cite{L2010,BM2022} and references therein.  Some preliminary analysis for a tenth-order extension of the equation has been made \cite{A-CEG2013,A-CEG2015}. 

Given the complications of increasing order it is appropriate to begin by examining equations in one space dimension.  To that end consider two-point boundary value problems of the form 
\begin{align}
f(x)&=(p_{m+1}(x)u^{(m+1)}(x))^{(m+1)}+\sum_{k=1}^m(p_k(x,u(x), \ldots, u^{(k-1)}(x))u^{(k)}(x))^{(k)}+p_0(x,u(x)), \, m \ge 0, \label{eq:basicbvp} \\ 
p_{m+1}(x)& > 0, \, x \in \Omega \equiv [a,b], \label{eq:basicbvpdomain} \\
u_{a,k}&=u^{(k)}(a), \, u_{b,k}=u^{(k)}(b), \, k=0, 1, \ldots, m, \label{eq:basicbc}
\end{align}
\noindent
where the functions $f$, $p_k, \, k=0, 1, \ldots m+1$, and constants $u_{a,k}, \, u_{b,k}, \, k=0, 1, \ldots m$, are such that solutions $u(x) \in C^{\infty}(\Omega)$ exist and are unique.  The weak form of \eqref{eq:basicbvp}, with $u(x) \in H^{(m+1)}(\Omega)$, is given by
\begin{align}
\int_a^b&f(x)v(x)dx=\int_a^b(-1)^{m+1}u^{(m+1)}(x)v^{(m+1)}(x)dx \nonumber \\
&+\int_{a}^b\Bigg(\sum_{k=1}^m(-1)^kp_k(x,u(x), \ldots, u^{(k-1)}(x))u^{(k)}(x)v^{(k)}(x) 
+p_0(x,u(x))v(x)\Bigg)dx, \, \forall v(x) \in H_0^{m+1}(\Omega). 
\end{align}
\noindent
The finite element solution is obtained by approximating $H_0^{m+1}(\Omega)$ and $H^{m+1}(\Omega)$ by finite dimensional subspaces.  The choice of finite dimensional subspace determines the accuracy of the solution.  The results shown extend to time-dependent problems and to equations on tensor-product domains in two and three space dimensions.  

This approach has some advantages over the alternative that replaces the high-order equation with a coupled system of lower-order equations.  If the goal is to control the error in $H^{m+1}(\Omega)$ it is more natural to leave the equation in its original form.   The same is true regarding the distribution of the additional degrees of freedom if a basis of order higher than $2m+1$ is employed.  For example, if the equation is replaced by a coupled system of second order equations and the basis order remains $p > 2m+1$ across all equations then the system contains an additional $(p-2m-1)(m+1)$ unknowns per element.

One of the difficulties in solving \eqref{eq:basicbvp} directly using finite elements is the determination of nodal basis functions to preserve $C^m$ continuity across element boundaries.  Linear and cubic Hermite splines are an obvious and common choice for $m=0$ and $m=1$, respectively.   These splines have been extended from $m=2$ to $m=11$ \cite{AK2016,H2012,JCY2017,KK2023,SA2007} by deriving the coefficients from the continuity conditions.  A general formula for any order can be derived from \cite{S1960}.  An alternative approach that involves one polynomial per interval based at its midpoint which matches solutions and $m+1$ derivatives at the element boundaries is employed in \cite{GHL2005} in a staggered grid solver for hyperbolic equations.

The complementary interior basis functions are also presented along with their key properties. These interior basis functions are formed from the product of lower-order ultraspherical polynomials multiplied by binomial powers of degree $m+1$ associated with the element boundaries.  They turn out to be natural extensions of Jacobi polynomials. Finally, it is shown that the two families of polynomials are linked through the derivatives of the Hermite splines.  

The building blocks of Section 2 are used to construct the Hermite-Jacobi interpolant in Section 3.  This is first carried out on a single interval $\Delta=[M-h/2,M+h/2]$.  Interpolation is done by matching the underlying function and its first $m$ derivatives at $M-h/2$ and $M+h/2$ and at the scaled roots of the paired ultraspherical polynomial.  A formula for the error in the interpolant is derived and extended to errors in $L^2(\Delta)$ and $H^k(\Delta), \, k=1, 2, \ldots m+1$ norms.  In proving the error formula, generalizations of the Lagrange interpolants are introduced.  The Hermite-Jacobi interpolant is extended to a grid that discretizes the interval $\Omega=[a,b]$ along with the $L^2(\Omega)$ and $H^k(\Omega)$ errors.  Superconvergence and the asymptotic equivalence of the Hermite-Jacobi interpolant are proved for a linear equation in Section 4.  Computations for a model problem demonstrate the theory in Section 5 with conclusions and next steps presented in Section 6.  

\newtheorem{theorem}{Theorem}[section]
\newtheorem{lemma}[theorem]{Lemma}
\newtheorem{corollary}{Corollary}[theorem]

\section{Arbitrary-Order Hermite and Jacobi Polynomials: Their Basis Properties and Interrelationships}
Consider the two families of polynomials on the two canonical intervals $\Omega_0=[0,1]$ and $\Omega_1=[-1,1]$, respectively.  On $\Omega_0$ consider the pair of polynomials of order $2m+1$,
\begin{align}
{\cal H}_l^m(\xi)&=\frac{(1-\xi)^{m+1}\xi^l}{l!} \sum_{n=0}^{m-l} S_n^m \xi^n, \, \label{eq:hermitedef0} \\
{\cal I}_l^m(\xi)&=(-1)^l{\cal H}_l^m(1-\xi), \, l=0, 1, \ldots m, \, \xi \in \Omega_0, \label{eq:hermitedef1}
\end{align}
\noindent
where
\begin{equation}
S_n^m=\frac{(m+n)!}{m!n!} \label{eq:defsmn}.
\end{equation}
\noindent
These pairs of polynomials for each $m$ serve as the left and right nodal basis functions, respectively.  As will be shown below they ensure $C^{m+1}$ continuity across element boundaries. On $\Omega_1$ consider the polynomials of order $p$ with zeros of order $m+1$ at $-1$ and $1$,
\begin{equation}
{\cal J}_p^m(\mu)=J_{m,p}(\mu^2-1)^{m+1}P_{p-m-1}^{(m+1)}(\mu)\equiv(\mu^2-1)^{m+1}{\cal J}_{p-2m-2}^{-m-2}(\mu), \, p \ge 2m+2, \, m \ge -1, \, \mu \in \Omega_1, \label{eq:ultrasphericaldef}
\end{equation}
\noindent
where $P_{p-m-1}(\mu)$ is the Legendre polynomial of degree $p-m-1$ and
\begin{equation}
J_{m,p}=\frac{2^{p-m-2}(p-m-2)!(p-2m-2)!}{(2p-2m-3)!}, \, p \ge 2m+2, \label{eq:h2}
\end{equation}
\noindent
is chosen so that ${\cal J}_p^m(\mu)$ is monic. It also follows that 
\begin{equation}
{\cal J}_{p-2m-2}^{-m-2}(\mu) \equiv J_{m,p}P_{p-m-1}^{(m+1)}(\mu). \label{eq:JNormalize}
\end{equation}
\noindent
The polynomials ${\cal J}_p^m(\mu)$, $m \ge 0$ serve as the interior basis functions on an element.  As will be shown, their counterparts, ${\cal J}_{p-2m-2}^{-m-2}(\mu)$ turn out to be the ultraspherical polynomials ${\cal C}_{p-2m-2}^{(\alpha)}(\mu)$ with $2\alpha=2m+3$ in monic form.  While both ${\cal J}_p^m(\mu)$ and ${\cal J}_{p-2m-2}^{-m-2}(\mu)$ form orthogonal sets of polynomials, the latter is a complete set while the former is not except when $m=-1$ when they are equal and equal to Legendre polynomials in monic form.

The key properties of these two families are laid out in the next two theorems.  Their relationship is presented in the third theorem.

\begin{theorem} \label{continuity}
The ${\cal H}_l^m(\xi)$, ${\cal I}_l^m(\xi)$ satisfy the Hermite spline continuity conditions
\begin{equation} \label{eq:h1}
({\cal H}_l^m(1))^{(k)}=({\cal I}_l^m(0))^{(k)}=0, \, k=0, 1, \ldots, m,
\end{equation}
\begin{equation} \label{eq:h1a}
({\cal H}_l^m(0))^{(k)}=({\cal I}_l^m(1))^{(k)} = \left\{
\begin{array}{cc}
1 & k=l, \\
0 & k \neq l, \\
\end{array}
\right., \, k = 0,1, \ldots m.
\end{equation}
\end{theorem}
\noindent
\begin{proof}
The proof follows directly from \cite{S1960} (see Theorem 1 with $n=1$, $x_0=0, \, x_1=1$, \, $r_0=r_1=m$) by computing $g_0^{(t)}(0)$ and $g_1^{(t)}(1)$.
\end{proof}

The structure of these Hermite splines is best illustrated by considering two examples of the ${\cal H}_l^m(\xi)$.
For the first example let $m=1$.  Then
\begin{equation}
{\cal H}_0^1(\xi)=(1-\xi)^2(1+2\xi), \, {\cal H}_1^1(\xi)=\xi(1-\xi)^2, \nonumber
\end{equation}
are the traditional cubic Hermite basis pair on [0,1] at the left end.
The second, with $m=5$, produces
\begin{align}
{\cal H}_0^{5}(\xi)=(1-\xi)^6(252\xi^5+126\xi^4+56\xi^3+21\xi^2+6\xi+1), \nonumber \\
{\cal H}_1^{5}(\xi)=\xi(1-\xi)^6(126\xi^4+56\xi^3+21\xi^2+6\xi+1), \nonumber \\
{\cal H}_2^{5}(\xi)=\xi^2(1-\xi)^6(56\xi^3+21\xi^2+6\xi+1)/2!, \nonumber \\
{\cal H}_3^{5}(\xi)=\xi^3(1-\xi)^6(21\xi^2+6\xi+1)/3!, \nonumber \\
{\cal H}_4^{5}(\xi)=\xi^4(1-\xi)^6(6\xi+1)/4!, \nonumber \\
{\cal H}_5^{5}(\xi)=\xi^5(1-\xi)^6/5!. \nonumber
\end{align}
\noindent
As $l$ increases one additional term is removed from the sum and the result is divided by $l$, that is,
\begin{equation}
{\cal H}_l^m(\xi)=({\cal H}_{l-1}^m(\xi)-S_{m-l+1}^m\xi^{m-l+1})\xi/l, \, l > 0. \label{eq:hermiterecursion}
\end{equation}
\noindent
In \cite{StoerBulirsch} (see section 2.1.5, p. 53) there is a similar recursion formula but in reverse, that is, starting with ${\cal H}_m^m(\xi)$ but unlike \eqref{eq:hermiterecursion} it involves the computation of a derivative of an auxiliary polynomial and thus is not explicit.

Having demonstrated the basis properties of the Hermite splines consider now the second family defined by \eqref{eq:ultrasphericaldef}-\eqref{eq:h2}.
\begin{theorem} \label{bubble}
The polynomials ${\cal J}_p^m(\mu)$ of \eqref{eq:ultrasphericaldef}-\eqref{eq:h2} and their counterparts ${\cal J}_p^{-m-2}(\mu)$ satisfy
\begin{align}
0&=({\cal J}_p^m(\pm 1))^{(k)}, \, k=0,1, \ldots, m, \, m \ge -1, \, p \ge 0, \label{eq:phi7a1} \\
&({\cal J}_{p-2m-1}^{-m-1}(\mu))^{\prime}=(p-2m-1){\cal J}_{p-2m-2}^{-m-2}(\mu), \, m \ge 0, \, p \ge 0, \label{eq:jderiv} \\
&{\cal J}_{p}^{m}(\mu)=p\int_{-1}^{\mu}{\cal J}_{p-1}^{m-1}(s)ds, \, m \ge 0, \, p \ge 2m+2, \label{eq:jint} \\
&{\cal J}_p^m(\mu)=
{\cal J}_p^{m-1}(\mu)-\frac{p(p-1)}{(2p-2m-1)(2p-2m-3)}{\cal J}_{p-2}^{m-1}(\mu), \, m \ge 0, \, p \ge 2m+2, \label{eq:phi4} \\
&{\cal J}_{p}^{-m-2}(\mu)=\frac{(2\alpha-1)!p!2^{p-1}((2p+2\alpha-3)/2)!}{(2p+2\alpha-2)!((2\alpha-1)/2)!}{\cal C}_{p}^{(\alpha)}, m \ge -1, \, p \ge 0, \, 2\alpha=2m+3,\label{eq:ultrasphericalequiv} \\
&{\cal J}_p^m(\mu)=\mu{\cal J}_{p-1}^m(\mu)-\frac{(p-1)(p-2m-3)}{(2p-2m-3)(2p-2m-5)}{\cal J}_{p-2}^m(\mu), \, {\cal J}_0^m(\mu)=1, \, {\cal J}_1^m(\mu)=\mu, \nonumber \\ 
m&=0, \pm 1, \pm 2, \ldots, \, p \ge 0, \label{eq:threetermrecurrence} \\
&\int_{-1}^1 \frac{{\cal J}_p^m(\mu){\cal J}_q^m(\mu)}{(1-\mu^2)^{m+1}}d\mu=\begin{cases}
0 & p \ne q, \, p,q \ge 2m+2, \, m \ge -1,\\
\frac{2^{2p-2m-3}((p-m-2)!)^2(p-2m-2)!p!}{(2p-2m-1)((2p-2m-3)!)^2} & p = q, \, p \ge 2m+2, \, m \ge -1. \label{eq:jortho}
\end{cases}
\end{align}
\end{theorem}
Two lemmas are needed before proceeding to the proof of the theorem.
\begin{lemma}
The Legendre polynomials satisfy
\begin{align}
0&=(1-\mu^2)P_p^{(m+2)}(\mu)-2(m+1)\mu P_{p}^{(m+1)}(\mu)+(p-m)(p+m+1)P_p^{(m)}(\mu), \label{eq:legendreexta}
\end{align}
\end{lemma}
\begin{proof}
From \cite{Rainville} (see equation 5 section 89)
\begin{equation}
(1-\mu^2)P_p^{\prime \prime}(\mu)-2\mu P_p^{\prime}(\mu)+p(p+1)P_p(\mu)=0,
\end{equation}
\noindent
and thus \eqref{eq:legendreexta} holds when $m=0$. By induction assume true for $m=0, 1, \ldots, {\bar M}$.  Then differentiating \eqref{eq:legendreexta} when $m={\bar M}$ implies that
\begin{align}
0&=(1-\mu^2)P_p^{({\bar M}+3)}(\mu)-2\mu P_p^{({\bar M}+2)}(\mu)-2({\bar M}+1)\mu P_p^{({\bar M}+2)}(\mu)-2({\bar M}+1)P_p^{({\bar M}+1)}(\mu) \nonumber \\
&+(p-{\bar M})(p+{\bar M}+1)P_p^{({\bar M}+1)}(\mu) \nonumber \\
&=(1-\mu^2)P_p^{({\bar M}+3)}(\mu)-2({\bar M}+2)\mu P_p^{({\bar M}+2)}(\mu)+(p^2+p-{\bar M}({\bar M}+1)-2({\bar M}+1))P_p^{({\bar M}+1)}(\mu) \nonumber \\
&=(1-\mu^2)P_p^{({\bar M}+3)}(\mu)-2({\bar M}+2)\mu P_p^{({\bar M}+2)}(\mu)+(p^2+p-({\bar M}+1)({\bar M}+2))P_p^{({\bar M}+1)}(\mu) \nonumber \\
&=(1-\mu^2)P_p^{({\bar M}+3)}(\mu)-2({\bar M}+2)\mu P_p^{({\bar M}+2)}(\mu)+(p-({\bar M}+1))(p+{\bar M}+2))P_p^{({\bar M}+1)}(\mu),
\label{eq:legendreextlemma1}
\end{align}
\noindent
thereby establishing \eqref{eq:legendreexta} when $m={\bar M}+1$.
\end{proof}
\begin{lemma}
The $J_{m,p}$ satisfy
\begin{equation}
J_{m-1,p-1}=(p-2m-1)J_{m,p}. \label{eq:jext}
\end{equation}
\end{lemma}
\begin{proof}
From \eqref{eq:h2}
\begin{equation}
\frac{J_{m-1,p-1}}{J_{m,p}}=\frac{\frac{2^{p-m-2}(p-m-2)!(p-2m-1)!}{(2p-2m-3)!}}{\frac{2^{p-m-2}(p-m-2)!(p-2m-2)!}{(2p-2m-3)!}}=p-2m-1.
\end{equation}
\end{proof}

\noindent
\textit{Proof (of Theorem \ref{bubble})}.  
Equation \eqref{eq:phi7a1} follows directly by differentiating \eqref{eq:h2}.  
From \eqref{eq:ultrasphericaldef}-\eqref{eq:h2} it follows that
\begin{equation}
J_{m-1,p-1}P_{p-m-1}^{(m)}(\mu)={\cal J}_{p-2m-1}^{-m-1}(\mu). \label{eq:h2a}
\end{equation}
\noindent
Differentiating \eqref{eq:h2a} and using \eqref{eq:ultrasphericaldef} and \eqref{eq:jext}
\noindent
leads to \eqref{eq:jderiv}.  To establish \eqref{eq:jint}, equation \eqref{eq:ultrasphericaldef} implies that
\begin{align}
({\cal J}_{p}^{m}(\mu))^{\prime}&=J_{m,p}(\mu^2-1)^{m+1}P_{p-m-1}^{(m+2)}(\mu)+J_{m,p}2\mu(\mu^2-1)^{m}(m+1)P_{p-m-1}^{(m+1)}(\mu) \nonumber \\
&=-J_{m,p}(\mu^2-1)^{m}((1-\mu^2)P_{p-m-1}^{(m+2)}(\mu)-2\mu(m+1)P_{p-m-1}^{(m+1)}(\mu)) \nonumber \\
&=J_{m,p}(\mu^2-1)(p)(p-2m-1)P_{p-m-1}^{(m)}(\mu) \nonumber \\
&=p{\cal J}_{p-1}^{m-1}(\mu), \label{eq:jintprime}
\end{align}
\noindent
where \eqref{eq:legendreexta} and \eqref{eq:jext} have been used. Equation \eqref{eq:jint} follows by integrating \eqref{eq:jintprime} together with \eqref{eq:phi7a1}.  From \cite{Rainville} (see equation 5, section 87)
\begin{equation}
(2p+1)P_p(\mu)=P_{p+1}^{\prime}(\mu)-P_{p-1}^{\prime}(\mu). \label{eq:legderiva}
\end{equation}
\noindent
Integrating \eqref{eq:legderiva} over $[-1,\mu]$ and using \eqref{eq:ultrasphericaldef} and \eqref{eq:jint} results in
\begin{equation}
\int_{-1}^{\mu}{\cal J}_p^{-1}(\rho)d\rho=\frac{1}{p+1}{\cal J}_{p+1}^{-1}(\mu)-\frac{p}{(2p+1)(2p-1)}{\cal J}_{p-1}^{-1}(\mu),
\end{equation}
\noindent
or
\begin{equation}
{\cal J}_{p+1}^0(\mu)={\cal J}_{p+1}^{-1}(\mu)-\frac{p(p+1)}{(2p+1)(2p-1)}{\cal J}_{p-1}^{-1}(\mu). \label{eq:legderivb}
\end{equation}
\noindent
Replacing $p+1$ with $p$ in \eqref{eq:legderivb} implies that \eqref{eq:phi4} holds when $m=0$.  By induction assume \eqref{eq:phi4} is true for $m=0, 1, \ldots, M$.  Integrating \eqref{eq:phi4} with $m=M$ yields
\begin{equation}
\int_{-1}^{\mu}{\cal J}_p^{\bar M}(\rho)d\rho=\int_{-1}^{\mu}{\cal J}_{p}^{{\bar M}-1}(\rho)d\rho-\frac{p(p-1)}{(2p-2{\bar M}-1)(2p-2{\bar M}-3)}\int_{-1}^{\mu}{\cal J}_{p-2}^{{\bar M}-1}(\rho)d\rho, \, p \ge 2{\bar M}+2,
\end{equation}
\noindent
or using \eqref{eq:jint}
\begin{equation}
\frac{1}{p+1}{\cal J}_{p+1}^{{\bar M}+1}(\mu)=\frac{1}{p+1}{\cal J}_{p+1}^{\bar M}(\mu)-\frac{p}{(2p-2{\bar M}-1)(2p-2{\bar M}-3)}{\cal J}_{p-1}^{\bar M}(\mu), \, p \ge 2{\bar M}+3. \label{eq:legderivc}
\end{equation}
Replacing $p+1$ by $p$ in \eqref{eq:legderivc} yields \eqref{eq:phi4} with $m={\bar M}+1$.  If ${\bar C}_p^{(\alpha)}(\mu)$ is the monic form of the ultraspherical polynomial then
\begin{equation}
{\bar {\cal C}}_p^{(\alpha)}(\mu)=\frac{(2\alpha-1)!p!2^{p-1}((2p+2\alpha-3)/2)!}{(2\alpha+2p-2)!((2\alpha-1)/2)!}{\cal C}_p^{(\alpha)}(\mu). \label{eq:monicultra}
\end{equation}
\noindent
follows from \cite{AS1972} (see equation 22.3.4 p. 775) along with the three-term recurrence relation
\begin{align}
{\bar {\cal C}}_{p+1}^{(\alpha)}(\mu)&=\mu{\bar {\cal C}}_p^{(\alpha)}(\mu)-\frac{p(p+2m+2)}{(2p+2m+3)(2p+2m+1)}{\bar {\cal C}}_{p-1}^{(\alpha)}(\mu), \, p \ge 1, \label{eq:ultrarecurrence} \\
{\bar {\cal C}}_0^{(\alpha)}&=1, \, {\bar {\cal C}}_1^{(\alpha)}=\mu, \label{eq:ultrarecurrenceinitial}
\end{align}
\noindent
\cite{AS1972} (see equation 22.7.3 p. 782).  Differentiating the three-term recurrence formula for Legendre polynomials in \cite{Rainville} (see equation 2 section 89) yields
\begin{equation}
pP_p^{\prime}(\mu)=(2p-1)P_{p-1}(\mu)+(2p-1)\mu P_{p-1}^{\prime}(\mu)-(p-1)P_{p-2}^{\prime}(\mu). \label{eq:legendrerecurrence}
\end{equation}
\noindent
Applying \eqref{eq:JNormalize} and \eqref{eq:jderiv} with $m=0$ to \eqref{eq:legendrerecurrence} leads to
\begin{equation}
\frac{p}{J_{0,p+1}}{\cal J}_{p-1}^{-2}(\mu)=(2p-1)P_{p-1}(\mu)+\mu \frac{2p-1}{J_{0,p}}{\cal J}_{p-2}^{-2}(\mu)-\frac{p-1}{J_{0,p-1}}{\cal J}_{p-3}^{-2}(\mu). \label{eq:jreducerecurrence}
\end{equation}
\noindent
From \cite{Rainville} (equation 6 section 87), \eqref{eq:jderiv}, and \eqref{eq:JNormalize} it is also true that
\begin{equation}
P_{p-1}(\mu)=\frac{1}{p}\left( \frac{1}{J_{0,p+1}}{\cal J}_{p-1}^{-2}(\mu)-\mu\frac{1}{J_{0,p}}{\cal J}_{p-2}^{-2}(\mu)\right). \label{eq:diffrecurrence}
\end{equation}
Substituting \eqref{eq:diffrecurrence} into \eqref{eq:jreducerecurrence} and using \eqref{eq:h2} gives
\begin{equation}
{\cal J}_{p+1}^{-2}(\mu)=\mu{\cal J}_{p}^{-2}(\mu)-\frac{p(p+2)}{(2p+3)(2p+1)}{\cal J}_{p-1}^{-2}(\mu),
\end{equation}
\noindent
which is the same as \eqref{eq:ultrarecurrence} with $m=0$ or $\alpha=3/2$.  Since \eqref{eq:h2} and \eqref{eq:JNormalize} imply that ${\cal J}_0^{-2}(\mu)=1$ and ${\cal J}_1^{-2}(\mu)=\mu$ it follows that \eqref{eq:ultrasphericalequiv} holds for $m=0$.  From \cite{OT2013}
\begin{equation}
{\bar {\cal C}}_p^{(\alpha) \prime}(\mu)=p{\bar {\cal C}}_{p-1}^{(\alpha+1)}. \label{eq:ultramonicderiv}
\end{equation}
\noindent
Then \eqref{eq:jderiv}, \eqref{eq:ultrasphericalequiv} with $m=0$, and \eqref{eq:ultramonicderiv} imply that \eqref{eq:ultrasphericalequiv} holds for $m > 0$ by successive differentiation.  This implies that \eqref{eq:threetermrecurrence} holds for $m \le -1$.  To complete the proof of \eqref{eq:threetermrecurrence} when $m > -1$ first note that from \eqref{eq:ultrasphericaldef} that the first two interior basis functions are
\begin{equation}
{\cal J}_{2m+2}^m(\mu)=(\mu^2-1)^{m+1}, \, {\cal J}_{2m+3}^m(\mu)=\mu(\mu^2-1)^{m+1}. \label{eq:firsttwointerior}
\end{equation}
\noindent
Integrating both sides of the three-term recurrence formula for Legendre polynomials in \cite{Rainville} (see equation 2, section 89), applying integration parts to the first term on the right, and using \eqref{eq:jint} and \eqref{eq:phi7a1} yields
\begin{align}
&p\int_{-1}^{\mu}P_p(\rho)d\rho=\frac{(2p-1)(2p-3)!}{2^{p-2}p!(p-2)!}\mu {\cal J}_p^0(\mu)-\frac{(2p-1)(2p-3)!}{2^{p-2}p!(p-2)!}\int_{-1}^{\mu}{\cal J}_p^0(\rho)d\rho \nonumber \\
&-(p-1)\int_{-1}^{\mu}P_{p-2}(\rho)d\rho, \, p \ge 2. \label{eq:phi1b}
\end{align}
\noindent
Applying \eqref{eq:phi4} with $m=0$ to the second integral in \eqref{eq:phi1b} results in
\begin{align}
&p\int_{-1}^{\mu}P_p(\rho)d\rho=\frac{(2p-1)(2p-3)!}{2^{p-2}p!((p-2)!}\mu {\cal J}_p^0(\mu) \nonumber \\
&-\frac{(2p-1)(2p-3)!}{2^{p-2}p!((p-2)!}\int_{-1}^{\mu}\left( {\cal J}_p^{-1}(\rho)-\frac{p(p-1)}{(2p-1)(2p-3)}{\cal J}_{p-2}^{-1}(\rho)\right)d\rho-(p-1)\int_{-1}^{\mu}P_{p-2}(\rho)d\rho, \, p \ge 2. \label{eq:phi1f}
\end{align}
\noindent
Combining like terms in \eqref{eq:phi1f} and applying \eqref{eq:jint} to the first and last terms together with \eqref{eq:ultrasphericaldef}-\eqref{eq:h2} and $m=-1$ yields
\begin{equation}
\frac{(2p-3)!(2p-1)}{2^{p-2}p!(p-2)!}{\cal J}_{p+1}^0(\mu)=\frac{(2p-1)(2p-3)!}{2^{p-2}p!((p-2)!}\mu{\cal J}_p^0(\mu)-\frac{(p-2)(2p-5)!}{2^{p-3}(p-1)!(p-3)!}{\cal J}_{p-1}^0(\mu), \, p \ge 3, \label{eq:phi1d}
\end{equation}
\noindent
where $p$ is now greater than 3 since ${\cal J}_2^0$ is the lowest degree interior polynomial that satisfies \eqref{eq:phi7a1}.  Dividing both sides of \eqref{eq:phi1d} by the coefficient of the first term gives
\begin{equation}
{\cal J}_{p+1}^0(\mu)=\mu{\cal J}_p^0(\mu)-\frac{2^{p-2}p!(p-2)!(p-2)(2p-5)!}{(2p-3)!(2p-1)2^{p-3}(p-1)!(p-3)!}{\cal J}_{p-1}^0(\mu), \, p \ge 3,
\end{equation}
\noindent
which simplifies to
\begin{equation}
{\cal J}_{p+1}^0(\mu)=\mu{\cal J}_p^0(\mu)-\frac{p(p-2)}{(2p-1)(2p-3)}{\cal J}_{p-1}^0(\mu), \, p \ge 3. \label{eq:phi1e}
\end{equation}
\noindent
Equation \eqref{eq:phi1e} is true for $p=2$ since in that case the second term on the right-hand side is zero and since \eqref{eq:firsttwointerior} holds.  When $p=1$, \eqref{eq:phi1e}, using \eqref{eq:firsttwointerior} can be written as
\begin{equation}
\mu{\cal J}_1^0(\mu)-{\cal J}_0^0(\mu)={\cal J}_2^0(\mu)\equiv(\mu^2-1),
\end{equation}
\noindent
which is satisfied if ${\cal J}_1^0(\mu)=\mu$ and ${\cal J}_0^0(\mu)=1$.  Thus, \eqref{eq:threetermrecurrence} holds for $m=0$.  By induction assume true for $m=0, 1, \ldots, {\hat M}$.  Integrating \eqref{eq:threetermrecurrence} and using integration by parts on the second term together with \eqref{eq:jint} results in
\begin{align}
&\frac{1}{p+1}{\cal J}_{p+1}^{{\hat M}+1}(\mu)=\frac{1}{p}\mu{\cal J}_p^{{\hat M}+1}(\mu)-\frac{1}{p}\int_{-1}^{\mu}{\cal J}_p^{{\hat M}+1}(\rho)d\rho
-\frac{(p-2{\bar M}-3)}{(2p-2{\bar M}-3)(2p-2{\bar M}-5)}{\cal J}_{p-1}^{{\hat M}+1}(\mu), \nonumber \\
&p \ge 2{\bar M}+5. \label{eq:phi1g}
\end{align}
\noindent
Substituting from \eqref{eq:phi4} for the integral in \eqref{eq:phi1g}, combining like terms, and finally applying \eqref{eq:jint} to the first and last terms yields
\begin{equation}
\left(1+\frac{1}{p} \right) \frac{1}{p+1}{\cal J}_{p+1}^{{\hat M}+1}(\mu)=\frac{1}{p}\mu{\cal J}_p^{{\hat M}+1}(\mu)-\frac{(p-2{\bar M}-4)}{(2p-2{\bar M}-3)(2p-2{\bar M}-5)}{\cal J}_{p-1}^{{\hat M}+1}(\mu),
\end{equation}
\noindent
or
\begin{equation}
{\cal J}_{p+1}^{{\hat M}+1}(\mu)=\mu{\cal J}_p^{{\hat M}+1}(\mu)-\frac{p(p-2{\bar M}-4)}{(2p-2{\bar M}-3)(2p-2{\bar M}-5)}{\cal J}_{p-1}^{{\hat M}+1}(\mu), \, p \ge 2{\bar M}+5,
\label{eq:phi1h}
\end{equation}
\noindent
which confirms the three-term recurrence relation for ${\hat M}+1$.  That \eqref{eq:phi1h} holds for $p=2{\bar M}+4$ follows from \eqref{eq:firsttwointerior} since the second term on the right-hand side is zero.  To show that ${\cal J}_1^{{\bar M}+1}(\mu)=\mu$ and ${\cal J}_0^{{\bar M}+1}(\mu)=1$ rewrite \eqref{eq:phi1h} with $p=2{\bar M}+3$ as
\begin{equation}
\mu{\cal J}_{2{\bar M}+3}^{{\bar M}+1}(\mu)+\frac{1}{2{\bar M}+1}{\cal J}_{2{\bar M}+2}(\mu)={\cal J}_{2{\bar M}+4}^{{\bar M}+1}(\mu)\equiv(\mu^2-1)^{{\bar M}+2}, \label{eq:phi1i}
\end{equation}
\noindent
where \eqref{eq:firsttwointerior} has been used.  Assume 
\begin{equation}
{\cal J}_p^{{\bar M}+1}(\mu)=\sum_{k=0}^p j_{p,k}^{{\bar M}+1}\mu^k, \, p \le 2{\bar M}+3. \label{eq:jacobiextn}
\end{equation}
\noindent
Plugging \eqref{eq:jacobiextn} into \eqref{eq:phi1i} and equating like powers of $\mu$ on both sides yields 
\begin{equation}
j_{p,p}^{{\bar M}+1}=1, \, j_{p,p-1}^{{\bar M}+1}=0. \label{eq:phi1j}
\end{equation}
\noindent
Substituting $p-1$ for $p$ in \eqref{eq:jacobiextn} and using \eqref{eq:phi1h} along with \eqref{eq:phi1j} implies that $j_{p-1,p-1}^{{\bar M}+1}=1$ and $j_{p-1,p-2}^{{\bar M}+1}=0$.  Induction on $p$ implies that ${\cal J}_1^{{\bar M}+1}(\mu)=\mu$.  Reversing directions then implies that ${\cal J}_p^{{\bar M}+1}(\mu)$ is odd or even as $p$ is odd or even, respectively.  From \eqref{eq:phi1i} it follows that $j_{p-1,0}^{{\bar M}+1}=(2{\bar M}+1)(-1)^{{\bar M}+2}$.  Continuing to move down in order and using the odd/even order of the ${\cal J}_p^{{\bar M}+1}(\mu)$ leads to
\begin{equation}
j_{2,0}^{{\bar M}+1}=(-1)^{{\bar M}+2}\frac{1\cdot(-1)\cdot(-3) \cdots (-2{\bar M}+1)}{3\cdot5\cdots(2{\bar M}-1)(2{\bar M}+1)}=\frac{1}{2{\bar M}+1},
\end{equation}
\noindent
which in turn implies that ${\cal J}_0^{{\bar M}+1}(\mu)=1$ using \eqref{eq:phi1h} with $p=2$.  Thus, \eqref{eq:threetermrecurrence} is established for all $m$.  To prove that the ${\cal J}_p^m(\mu)$ are orthogonal with respect to the weight function $(1-\mu^2)^{-(m+1)}$ from \eqref{eq:ultrasphericaldef} it follows that
\begin{equation}
\int_{-1}^1 \frac{{\cal J}_p^m(\rho){\cal J}_q^m(\rho)}{(1-\rho^2)^{m+1}}d\rho=\int_{-1}^1{\cal J}_{p-2m-2}^{-m-2}(\rho){\cal J}_{q-2m-2}^{-m-2}(\rho)(1-\rho^2)^{m+1}d\rho. \label{eq:jorthoa}
\end{equation}
\noindent
Substituting into \eqref{eq:jorthoa} from \eqref{eq:ultrasphericalequiv} yields
\begin{align}
\int_{-1}^1 \frac{{\cal J}_p^m(\rho){\cal J}_q^m(\rho)}{(1-\rho^2)^{m+1}}d\rho&=\frac{((2m+2)!)^2(p-2m-2)!(q-2m-2)!2^{p+q-4m-6}(p-m-2)!(q-m-2)!}{(2p-2m-3)!(2q-2m-3)!((m+1)!)^2} \nonumber \\
&\times \int_{-1}^1C_{p-2m-2}^{(\alpha)}C_{q-2m-2}^{(\alpha)}(1-\rho^2)^{\alpha-1/2}d\rho. \label{eq:orthob}
\end{align}
\noindent
This establishes \eqref{eq:jortho} in the case $p \ne q$ from the orthogonality of the ultraspherical polynomials.  Since from \cite{AS1972} (see equation 22.2.3 p. 774)
\begin{align}
\int_{-1}^1(C_{p-2m-2}^{(\alpha)}(\rho))^2(1-\rho^2)^{\alpha-1/2}d\rho&=\frac{\pi2^{1-2\alpha}\Gamma(p-2m-2+2\alpha)}{(p-2m-2)!(p-2m-2+\alpha)(\Gamma(\alpha))^2} \nonumber \\
&=\frac{2^{2m+1}p!(m!)^2}{(p-2m-2)!(2p-2m-1)((2m+1)!)^2},
\end{align}
\noindent
equation \eqref{eq:orthob} simplifies to
\begin{align}
\int_{-1}^1 \frac{({\cal J}_p^m(\rho))^2}{(1-\rho^2)^{m+1}}d\rho&=\frac{2^{2m+1}p!(m!)^2}{(p-2m-2)!(2p-2m-1)((2m+1)!)^2} \nonumber \\
&\times \frac{((2m+2)!)^2((p-2m-2)!)^22^{2p-4m-6}((p-m-2)!)^2}{((2p-2m-3)!)^2(m+1)!},
\end{align}
\noindent
which further simplifies to \eqref{eq:jortho}.
\qed

\noindent
Remark: The polynomials ${\cal J}_p^m(\mu)$ are natural extensions of the Jacobi polynomials $P_p^{(-m-1,-m-1)}(\mu)$, $m \ge 0$, and are therefore also extensions of the ultraspherical polynomials $C_p^{(\alpha)}, \, m \le -2,$ based on extending the three-term recurrence relation including the same two starting polynomials.

The Hermite splines and the polynomials ${\cal J}_p^m(\mu)$ defined by \eqref{eq:ultrasphericaldef}-\eqref{eq:h2} are connected as demonstrated in the next theorem.
\begin{theorem}
\begin{align}
({\cal H}_l^m(\xi))^{\prime}&=\begin{cases}
\frac{(-1)^{m+1}}{2^{2m}}\frac{(2m+1)!}{(m!)^2} {\cal J}_{2m}^{m-1}(2\xi-1) & l=0, \\
\frac{(-1)^{m+1}}{2^{2m}}\frac{(2m-l)!(2m+1)}{l!m!(m-l)!}{\cal J}_{2m}^{m-1}(2\xi-1)+{\cal H}_{l-1}^{m-1}(\xi) & l =1, 2, \ldots , m, \, m > 1. \label{eq:derivhermite}
\end{cases}
\end{align}
\end{theorem}
\begin{proof}
To prove \eqref{eq:derivhermite} consider first the case $l=0$.  Differentiating \eqref{eq:hermitedef0} together with \eqref{eq:defsmn} produces
\begin{align}
({\cal H}_0^m(\xi))^{\prime}&=-(m+1)(1-\xi)^m\sum_{n=0}^m\frac{(m+n)!}{m!n!}\xi^n+(1-\xi)^{m+1}\sum_{n=1}^m\frac{(m+n)!}{m!n!}n\xi^{n-1} \nonumber \\
&=(1-\xi)^m \left(-\sum_{n=0}^m\frac{(m+1)(m+n)!}{m!n!}\xi^n+\sum_{n=1}^m\frac{(m+n)!n}{m!n!}\xi^{n-1}-\sum_{n=1}^m\frac{(m+n)!n}{m!n!}\xi^n \right) \nonumber \\
&=(1-\xi)^m \left(-\sum_{n=0}^m\frac{(m+1)(m+n)!}{m!n!}\xi^n+\sum_{n=0}^{m-1}\frac{(m+n+1)!}{m!n!}\xi^{n}-\sum_{n=1}^m\frac{(m+n)!n}{m!n!}\xi^n \right) \nonumber \\
&=(1-\xi)^m\left(\sum_{n=1}^{m-1}\xi^n\frac{1}{m!n!}(-(m+1)(m+n)!+(m+n+1)!-n(m+n)!) \right. \nonumber \\
&\left.-\frac{(m+1)m!}{m!}+\frac{(m+1)!}{m!}-\frac{(m+1)(2m)!}{(m!)^2}\xi^m-\frac{(2m)!m}{(m!)^2}\xi^m\right) \nonumber \\
&=-\frac{(2m+1)!}{(m!)^2}(1-\xi)^m\xi^m \nonumber \\
&=\frac{(-1)^{m+1}(2m-1)!}{2^{2m}(m!)^2}{\cal J}_{2m}^{m-1}(2\xi-1),
\end{align}
\noindent
where \eqref{eq:ultrasphericaldef} has been employed.  For $l > 0$
\begin{align}
 ({\cal H}_l^m(\xi))^{\prime}&=\frac{1}{l!}\left(-(m+1)(1-\xi)^m\sum_{n=0}^{m-l}\frac{(m+n)!}{m!n!}\xi^{n+l}+(1-\xi)^{m+1}\sum_{n=0}^{m-l}\frac{(m+n)!}{m!n!}(n+l)\xi^{n+l-1} \right) \nonumber \\
 &=\frac{(1-\xi)^m}{m!l!}\left( -\sum_{n=0}^{m-l}\frac{(m+n)!}{n!}(m+1+n+l)\xi^{n+l}+\sum_{n=0}^{m-l}\frac{(m+n)!}{n!}(n+l)\xi^{n+l-1} \right) \nonumber \\
 &=\frac{(1-\xi)^m}{m!l!}\left( \sum_{n=1}^{m-l}\frac{(m+n-1)!}{(n-1)!}(-(m+n+l)n+(n+l)(m+n))\xi^{n+l-1}\right. \nonumber \\
 & \left.-\frac{(2m-l)!}{(m-l)!}(2m+1)\xi^{m}+m!l\xi^{l-1}\right) \nonumber \\
 &=\frac{(1-\xi)^m}{m!l!}\sum_{n=1}^{m-l}\frac{(m+n-1)!}{n!}ml\xi^{n+l-1}-\frac{(1-\xi)^m}{m!l!}\frac{(2m-l)!(2m+1)}{(m-l)!}\xi^m+\frac{(1-\xi)^m}{(l-1)!}\xi^{l-1} \nonumber \\
 &=\frac{(1-\xi)^m}{m!l!}\sum_{n=0}^{m-l}\frac{(m+n-1)!}{n!}ml\xi^{n+l-1}-\frac{(1-\xi)^m}{m!l!}\frac{(2m-l)!(2m+1)}{(m-l)!}\xi^m \nonumber \\
 &={\cal H}_{l-1}^{m-1}(\xi)+(-1)^{m+1}\frac{(2m-l)!(2m+1)}{2^{2m}m!l!(m-l)!}{\cal J}_{2m}^{m-1}(2\xi-1),
\end{align}
\noindent
where \eqref{eq:hermitedef0}, \eqref{eq:defsmn}-\eqref{eq:ultrasphericaldef} have been used.
\end{proof}

\section{The Hermite-Jacobi Interpolant}
The Hermite-Jacobi interpolant (the latter name is chosen since the ${\cal J}_p^m(\mu)$ are extensions of Jacobi polynomials) is constructed from the two families of polynomials introduced in the previous section.  The Hermite splines, \eqref{eq:hermitedef0}-\eqref{eq:defsmn}, serve as the basis functions at the nodes, while the subfamily of the extended Jacobi polynomials, \eqref{eq:ultrasphericaldef}-\eqref{eq:h2}, are the basis functions in the interior.  Consider the interval $\Delta\equiv[M-h/2,M+h/2]$ and let $u(x) \in C^{\infty}({\bar \Delta})$. The approximation space $S^{p,\Delta}$ consists of polynomials of degree $p \ge 2m+1$.  To that end define
\begin{align}
\Phi_{i,L,\Delta}^m(x)&={\cal H}_i^m((x-(M-h/2))/h), \, i=0, \cdots , m, \label{eq:leftspline} \\
\Phi_{i,R,\Delta}^m(x)&={\cal I}_i^m((x-(M-h/2))/h), \, i=0, \cdots , m, \label{eq:rightspline} \\
\Upsilon_{i,\Delta}^m(x)&=\left(\frac{h}{2}\right)^{i}{\cal J}_i^m(2(x-M)/h), \, i=2m+2, 2m+3, \ldots, p, \label{eq:defUpsilon} \\
\Upsilon_{p+1,\Delta}^m(x)&=\left(\frac{h}{2}\right)^{p+1}{\cal J}_{p+1}^m(2(x-M)/h)\equiv\sum_{i=0}^{p+1}{\hat a}_{p+1-2i}^{p+1,m}(x-M)^{p+1-2i}, \, {\hat a}_{p+1}^{p+1,m}=1, \label{eq:defpsi} \\
\Upsilon_{p-2m-1,\Delta}^{-m-2}(x)&=\left(\frac{h}{2}\right)^{p-2m-1}{\cal J}_{p-2m-1}^{-m-2}(2(x-M)/h), \label{eq:defpsia}
\end{align}
\noindent
where from \eqref{eq:ultrasphericaldef}
\begin{equation}
\Upsilon_{p+1,\Delta}^m(x)=(x-M+h/2)^{m+1}(x-M-h/2)^{m+1}\Upsilon_{p-2m-1,\Delta}^{-m-2}(x), \label{eq:upsilonp1a}
\end{equation}
\noindent
and the $\Upsilon_{i,\Delta}^m(x)$, $i=2m+2, 2m+3, \ldots, p+1$ are monic.  Further, let
\begin{align} 
W_{p,\Delta}^m(x)&=\begin{cases}
\sum_{i=0}^m(W_{i,L}\Phi_{i,L,\Delta}^m(x)+W_{i,R}\Phi_{i,R,\Delta}^m(x)) & p=2m+1, \\
\sum_{i=0}^m(W_{i,L}\Phi_{i,L,\Delta}^m(x)+W_{i,R}\Phi_{i,R,\Delta}^m(x))+\sum_{i=2m+2}^p W_{i,M}\Upsilon_{i,\Delta}^m(x) & p \ge 2m+2,
\end{cases} \label{eq:interpolant}
\end{align}
be an interpolant on $\Delta$ such that
\begin{align}
W_{i,L}&=u^{(i)}(M-h/2), \, i=0, 1, \ldots m,   \label{eq:left} \\
W_{i,R}&=u^{(i)}(M+h/2), \, i=0, 1, \ldots m, \, \label{eq:right}
\end{align}
 and, if $p \ge 2m+2$, the $W_{i,M}$ are chosen so that
\begin{equation}
W_{p,\Delta}^m(M_i)=u(M_i), \, i=2, 3, \ldots, p-2m, \label{eq:middle}
\end{equation}
with $M_i=M+h\mu_i/2$ where $\mu_i$ are the $p-2m-1$ roots of $\Upsilon_{p-2m-1,\Delta}^{-m-2}(\mu)$, $i=2, 3, \ldots, p-2m$ and $M_1=M-h/2$ and $M_{p-2m+1}=M+h/2$. 
\noindent
Additionally, let $\pi^{\Delta}$ be the linear operator that projects functions in $C^{\infty}({\bar \Delta})$ onto the space of polynomials of degree $p$ on ${\bar \Delta}$ according to \eqref{eq:interpolant}-\eqref{eq:middle} and let
\begin{align}
v(x)=&\sum_{i=0}^{p+s+1} \frac{(x-M)^i}{i!}u^{(i)}(M), \label{eq:vdefinition} \\
{\hat L}_{i,\Delta}^{m,p}(x)=&\frac{(x-M_1)^{m+1}}{(M_i-M_1)^{m+1}}\frac{(x-M_{p-2m+1})^{m+1}}{M_i-M_{p-2m+1})^{m+1}}\prod_{j=2,j \neq i}^{p-2m}\frac{(x-M_j)}{(M_i-M_j)}, i=2, 3, \ldots, p-2m, \label{eq:lhatdefinition} \\
{\hat A}_{l,\Delta}^{m,p}(x)=&\frac{\sum_{j=0}^{m-l}A_{j,l}(x-M_1)^{j+l}}{l!}\frac{(x-M_{p-2m+1})^{m+1}\prod_{j=2}^{p-2m}(x-M_j)}{(M_1-M_{p-2m+1})^{m+1}\prod_{j=2}^{p-2m}(M_1-M_j)}, \label{eq:ahatdefinition} \\
{\hat B}_{l,\Delta}^{m,p}(x)=&\frac{\sum_{j=0}^{m-l}B_{j,l}(x-M_{p-2m+1})^{j+l}}{l!}\frac{(x-M_1)^{m+1}\prod_{j=2}^{p-2m}(x-M_j)}{(M_{p-2m+1}-M_1)^{m+1}\prod_{j=2}^{p-2m}(M_{p-2m+1}-M_j)}, \label{eq:bhatdefinition} \\
l&=0, 1, \ldots, m,
\end{align}
where the $A_{j,l}$ and $B_{j,l}$ are determined by requiring that
\begin{align}
({\hat A}_l^{m,p}(M_1))^{(k)} &=
\begin{cases}
1 & l=k, \\
0 & l \neq k, \, k=0, 1, \ldots, m,
\end{cases}  \label{eq:Arequirements} \\
({\hat B}_l^{m,p}(M_{p-2m+1}))^{(k)}&=
\begin{cases}
1 & l=k, \\
0 & l \neq k, \, k=0, 1, \ldots, m.
\end{cases} \label{eq:Brequirements}
\end{align}
\noindent
Remark: The number of conditions \eqref{eq:Arequirements}-\eqref{eq:Brequirements} is larger than the number of coefficients $A_{j,l}$ and $B_{j,l}$, respectively.  That unique values of the coefficients can be found while simultaneously meeting the conditions is demonstrated below.

\begin{theorem}
\label{interpolation}
If \eqref{eq:leftspline}-\eqref{eq:Brequirements} hold and if $u(x) \in C^{\infty}({\bar \Delta})$ then $\forall x \in (M-h/2,M+h/2) \, \exists \, {\hat x} \in \Delta$ s.t. for any $s \in \mathbb{N}_0$
\begin{align}
u(x)-&W_{p,\Delta}^m(x)=\Upsilon_{p+1,\Delta}^m(x)\sum_{i=0}^s \frac{u^{(p+i+1)}(M)}{(p+i+1)!}\Psi_{i,\Delta}^m(x) \nonumber \\
&+\frac{u^{(p+s+2)}(M)}{(p+s+2)!}({\hat x}-M)^{p+s+2}+\sum_{i=2}^{p-2m}(v(M_i)-u(M_i)) {\hat L}_{i,\Delta}^{m,p}(x) \nonumber \\
&+\sum_{l=0}^{m}(v^{(l)}(M_1)-u^{(l)}(M_1)){\hat A}_{l,\Delta}^{m,p}(x)+\sum_{l=0}^{m}(v^{(l)}(M_{p-2m+1})-u^{(l)}(M_{p-2m+1})){\hat B}_{l,\Delta}^{m,p}(x), \label{eq:theorem1}
\end{align}
where
\begin{align}
\Psi_{i,\Delta}^m(x)=&\frac{(x-M)^{p+i+1}-\pi^{\Delta}(x-M)^{p+i+1}}{\Upsilon_{p+1,\Delta}^m(x)} \label{eq:psidefinitiona} \\
=&\sum_{l=0}^{\lfloor i/2 \rfloor}a_{i-2l}^{i,m}(x-M)^{i-2l}, i=1, 2, \ldots, s, \label{eq:psidefinitionb} \\
a_i^{i,m}&=1, \, a_{i-2l}^{i,m}=-\sum_{j=1}^{min(l,\lfloor (p+1)/2 \rfloor )}a_{i-2(l-j)}^{i,m}{\hat a}_{p+1-2j}^{p+1}, l=1, 2, \ldots , \lfloor i/2 \rfloor . \label{eq:psidefinitionc} 
\end{align}
\end{theorem}

\noindent
Remark: If $m=0$ the ${\hat L}_{i,\Delta}^{m,p}(x)$, ${\hat A}_{l,\Delta}^{m,p}(x)$, and ${\hat B}_{l,\Delta}^{m,p}(x)$ are the Lagrange interpolants.

Before proceeding to the proof of the theorem, four lemmas are proved that establish that the ${\hat L}_{i,\Delta}^{p,m}(x)$, ${\hat A}_{l,\Delta}^{p,m}(x)$, ${\hat B}_{l,\Delta}^{p,m}(x)$ form a basis for $S^{p,\Delta}$.  Thus, they constitute generalizations of the Lagrange interpolating polynomials.  One further lemma demonstrates a key property of $\Upsilon_{p+1,\Delta}^m(x)$.
\begin{lemma} \label{lhatlemma}
The ${\hat L}_{i,\Delta}^{m,p}(x)$, $i=2, 3, \ldots, p-2m$, of \eqref{eq:lhatdefinition} satisfy
\begin{align}
{\hat L}_{i,\Delta}^{m,p}(M_j)&=\begin{cases}
0 & j=2, 3, \ldots, p-2m, \, j \ne i, \\
1 & j=i,
\end{cases} \label{eq:lhat1}
\\
({\hat L}_{i,\Delta}^{m,p})^{(k)}(M_j)&=0, \, j=1,p-2m+1, \, k=1, 2, \ldots, m. \label{eq:lhat2}
\end{align}
\end{lemma}
\begin{proof}
The ${\hat L}_{i,\Delta}^{m,p}(x)$ are the Lagrange interpolants at the $M_j$, $j=2, 3, \ldots, p-2m-2,$ multiplied by the binomials $\frac{(x-M_j)^{m+1}}{(M_i-M_j)^{m+1}}$, $j=1, p-2m-1$, which together imply \eqref{eq:lhat1} and \eqref{eq:lhat2}.
\end{proof}
\noindent
To prove the comparable properties of the ${\hat A}_{l,\Delta}^{m,p}(x)$ and ${\hat B}_{l,\Delta}^{m,p}(x)$ it is helpful to decompose them into smaller components. To that end let
\begin{equation}
{\hat A}_{l,\Delta}^{m,p}(x)=F_{l,\Delta}^m(x)G_{p,\Delta}^m(x), \label{eq:hataproduct}   
\end{equation}
\noindent
where
\begin{align} 
F_{l,\Delta}^m(x)&=\sum_{j=0}^{m-l}A_{j,l}(x-M_1)^{j+l}, \label{eq:defatilde} \\
G_{p,\Delta}^m(x)&={\bar G}_{p,\Delta}^m(x-M_{p-2m+1})^{m+1}\prod_{j=2}^{p-2m}(x-M_j), \label{eq:defabara} \\
{\bar G}_{p,\Delta}^m&=\frac{1}{(M_1-M_{p-2m+1})^{m+1}\prod_{j=2}^{p-2m}(M_1-M_j)}, \label{eq:defgbar} \\
{\hat G}_{p,\Delta}^m(x)&=\frac{m+1}{x-M_{p-2m+1}}+\sum_{j=2}^{p-2m}\frac{1}{x-M_j}, \label{eq:acheck}
\end{align}
\noindent
with analogous functions for ${\hat B}_{l,\Delta}^{m,p}.$

\begin{lemma}
The following relationships hold for the functions defined in \eqref{eq:defatilde}-\eqref{eq:acheck}:
\begin{align}
({\hat G}_{p,\Delta}^m(M_1))^{(k)}&=(-1)^kk!\left(\frac{m+1}{(M_1-M_{p-2m+1})^{k+1}}+\sum_{j=2}^{p-2m}\frac{1}{(M_1-M_j)^{k+1}}\right), \nonumber \\
k&=0, 1, \ldots, m-1,\label{eq:achecka} \\
(F_{l,\Delta}^m(M_1))^{(i)}&=\begin{cases}
0 & i < l, \\
A_{i-l,l}i! & i \ge l,
\end{cases} \label{eq:atildea} \\
G_{p,\Delta}^m(M_j)&=0, \, j =2, 3, \ldots, p-2m+1, \label{eq:zero} \\
(G_{p,\Delta}^m(M_{p-2m+1}))^{(k)}&=0, k=0, 1, \ldots, m, \label{eq:derivzero} \\
G_{p,\Delta}^m(M_1)&= 1, \label{eq:nonzero} \\
(G_{p,\Delta}^m(M_1))^{(k)}&=\sum_{i=0}^{k-1}\binom{k-1}{i}(G_{p,\Delta}^m(M_1))^{(i)}({\hat G}_{p,\Delta}^m(M_1))^{(k-1-i)}, \, k=1, 2, \ldots, m. \label{eq:abarderiv1}
\end{align}
\end{lemma}
\begin{proof}
Since from \eqref{eq:acheck}
\begin{equation}
({\hat G}_{p,\Delta}^m(x))^{(k)}=(-1)^kk!\left(\frac{m+1}{(x-M_{p-2m+1})^{k+1}}+\sum_{j=2}^{p-2m}\frac{1}{(x-M_j)^{k+1}}\right),    
\end{equation}
\noindent
equation \eqref{eq:achecka} follows directly.  From \eqref{eq:defatilde}
\begin{equation}
(F_{l,\Delta}^m(x))^{(i)}=
\sum_{j={\rm max}(0,i-l)}^{m-l}A_{j,l}\frac{(j+l)!}{(j+l-i)!}(x-M_1)^{j+l-i}. \label{eq:ahatderivab}
\end{equation}
\noindent
When $i < l$ and $x=M_1$, \eqref{eq:ahatderivab} implies that
\begin{equation}
(F_{l,\Delta}^m(M_1))^{(i)}=0, \label{eq:atildederivaa}
\end{equation}
\noindent
since for each term in the sum, $(x-M_1)^{j+l-i}$, $j+l-i > 0$.  When $i \ge l$, the term $(x-M_1)^{j+l-i}$ in \eqref{eq:ahatderivab} is non-zero if $x=M_1$ only when $j=i-l$ which implies that
\begin{equation}
(F_{l,\Delta}^m(M_1))^{(i)}=A_{i-l,l}i!. \label{eq:atildederivb}
\end{equation}
\noindent
Together \eqref{eq:atildederivaa} and \eqref{eq:atildederivb} establish \eqref{eq:atildea}. Equations \eqref{eq:zero}-\eqref{eq:nonzero} follow directly from \eqref{eq:defabara}.  Taking the log of both sides of \eqref{eq:defabara} and differentiating yields
\begin{equation}
\frac{(G_{p,\Delta}^m(x))^{\prime}}{G_{p,\Delta}^m(x)}={\hat G}_{p,\Delta}^m(x). \label{eq:lnb}
\end{equation}
\noindent
Substituting \eqref{eq:nonzero} into \eqref{eq:lnb} yields \eqref{eq:abarderiv1} with $k=1$.  Differentiating \eqref{eq:lnb} $k-1$ times gives
\begin{equation}
(G_{p,\Delta}^m(x))^{(k)}=\sum_{i=0}^{k-1}\binom{k-1}{i}(G_{p,\Delta}^m(x))^{(i)}({\hat G}_{p,\Delta}^m(x))^{(k-1-i)}, \, k \ge 1. \label{eq:abarderiv2}
\end{equation}
\noindent
Substituting \eqref{eq:nonzero} into \eqref{eq:abarderiv2} for each $k$ yields \eqref{eq:abarderiv1}.
\end{proof}
\noindent
\begin{lemma} \label{ahatlemma}
Based on these results the ${\hat A}_{l,\Delta}^{m,p}(x)$, and ${\hat B}_{l,\Delta}^{m,p}(x)$, $l=0, 1, \ldots m$, of \eqref{eq:ahatdefinition}-\eqref{eq:Brequirements} satisfy
\begin{align}
{\hat A}_{l,\Delta}^{p,m}(M_j)&=0, j=2, 3, \ldots, p-2m, \, l =0, 1, \ldots m, \\
({\hat A}_{l,\Delta}^{p,m})^{(k)}(M_j)&=\begin{cases}
0 & j=p-2m+1, \, k=0, 1, \ldots, m, \, l=0, 1, \ldots, m, \\
0 & j=1, \, k=0, 1, \ldots, m, \, l=0, 1, \ldots, m, \, k \ne l, \\
1 & j=1, \, k=0, 1, \ldots, m, \, l=0, 1, \ldots, m, \, k=l,
\end{cases}
\end{align}
\begin{align}
{\hat B}_{l,\Delta}^{p,m}(M_j)&=0, j=2, 3, \ldots, p-2m, \, l =0, 1, \ldots m, \\
({\hat B}_{l,\Delta}^{p,m})^{(k)}(M_j)&=\begin{cases}
0 & j=1, \, k=0, 1, \ldots, m, \, l=0, 1, \ldots, m, \\
0 & j=p-2m+1, \, k=0, 1, \ldots, m, \, l=0, 1, \ldots, m, \, k \ne l, \\
1 & j=p-2m+1, \, k=0, 1, \ldots, m, \, l=0, 1, \ldots, m, \, k=l,
\end{cases}
\end{align}
\end{lemma}
\begin{proof}
The first step is to show that the conditions \eqref{eq:Arequirements} can be satisfied.  It immediately follows from \eqref{eq:hataproduct} and \eqref{eq:zero} that
\begin{equation}
{\hat A}_{l,\Delta}^{m,p}(M_j)=0, \, j=2, 3, \ldots M_{p-2m}.
\end{equation}
\noindent
Differentiating \eqref{eq:hataproduct} $k$ times yields
\begin{equation}
({\hat A}_l^{m,p}(x))^{(k)}=\frac{1}{l!}\sum_{i=0}^k\binom{k}{i}(F_{l,\Delta}^m(x))^{(i)}(G_{p,\Delta}^m(x))^{(k-i)}, \, k \ge 0.\label{eq:ahatderivaa}
\end{equation}
\noindent
Then \eqref{eq:derivzero} and \eqref{eq:ahatderivaa} imply that
\begin{equation}
({\hat A}_{l,\Delta}^{m,p}(M_{p-2m+1}))^{(k)}=0, \, k=0, 1, \ldots, m, \, l=0, \ldots m.
\end{equation}
\noindent
Thus, for fixed $l$, $p-m$ conditions are automatically satisfied. 
Furthermore, when $k < l$, \eqref{eq:atildea} and \eqref{eq:ahatderivaa} imply that
\begin{equation}
({\hat A}_l^{m,p}(M_1))^{(k)}=0.
\end{equation}
\noindent
Hence, an additional $l$ conditions are automatically satisfied leaving $m-l+1$ conditions for the $m-l+1$ coefficients $A_{j,l}$.  The following algorithm confirms that the remaining conditions suffice to determine these coefficients.  
\noindent
Combining \eqref{eq:atildea} and \eqref{eq:ahatderivaa} when $k \ge l$ yields
\begin{align}
({\hat A}_l^{m,p}(M_1))^{(k)}&=\frac{1}{l!}\sum_{i=0}^{l-1}(F_{l,\Delta}^m(M_1))^{(i)}(G_{p,\Delta}^m(M_1))^{(k-i)}+\frac{1}{l!}\sum_{i=l}^{k}(F_{l,\Delta}^m(M_1))^{(i)}(G_{p,\Delta}^m(M_1))^{(k-i)} \nonumber \\
&=\frac{1}{l!}\sum_{i=l}^{k}(F_{l,\Delta}^m(M_1))^{(i)}(G_{p,\Delta}^m(M_1))^{(k-i)} \nonumber \\
&=\frac{1}{l!}\sum_{i=l}^k \binom{k}{i}A_{i-l,l}i!(G_{p,\Delta}^m(M_1))^{(k-i)}. \label{eq:ahatderivf}
\end{align}
\noindent
Using \eqref{eq:ahatderivf} with $k=l$, together with \eqref{eq:Arequirements} and \eqref{eq:nonzero} it follows that
\begin{equation}
A_{0,l}=1. \label{eq:AJLa}
\end{equation}
\noindent
Solving \eqref{eq:ahatderivf} successively for $A_{j-l,l}$, $j=l+1, l+2, \ldots k$, using \eqref{eq:nonzero} at each step yields
\begin{equation}
A_{j-l,l}=-\sum_{i=l}^{j-1}\frac{A_{i-l,l}(G_{p,\Delta}^m(M_1))^{(j-i)}}{(j-i)!}. \label{eq:AJLb}
\end{equation}
\noindent
By re-indexing with $i$ replaced by $i+l$ and $j$ replaced by $j+l$ \eqref{eq:AJLb} becomes
\begin{equation}
A_{j,l}=-\sum_{i=0}^{j-1}\frac{A_{i,l}(G_{p,\Delta}^m(M_1))^{j-i}}{(j-i)!}, \, j=1, 2, \ldots, m-l. \label{eq:AJLc}
\end{equation}
\noindent
Thus, the $A_{j,l}$ can be found successively from $j=0, 1, \ldots, m-l$, with \eqref{eq:abarderiv1} confirming that \eqref{eq:Arequirements} can be satisfied.  The $B_{j,l}$, which satisfy \eqref{eq:Brequirements}, can be handled in a comparable manner.  
\end{proof}
\begin{lemma} \label{basis}
The ${\hat L}_{i,\Delta}^{m,p}(x)$, $i=2, 3, \ldots, p-2m$, ${\hat A}_{l,\Delta}^{m,p}(x)$, and ${\hat B}_{l,\Delta}^{m,p}(x)$, $l=0, 1, \ldots, m,$ constitute a basis for $S^{p,\Delta}$.
\end{lemma}
\begin{proof}
The result follows directly from Lemmas \ref{lhatlemma} and \ref{ahatlemma}.
\end{proof}
\begin{lemma} \label{upsilonp1lemma}
The lowest degree monic polynomial $p(x)$ satisfying
\begin{align}
p(M_j)&=0, j=2, 3, \ldots, p-2m, \label{eq:upsilonzerosa} \\
(p(M_j))^{(k)}&=0, j=1, p-2m+1, \, k=0, 1, \ldots, m, \label{eq:upsilonzerosb}
\end{align}
\noindent
is $\Upsilon_{p+1,\Delta}^m(x)$.
\end{lemma}
\begin{proof}
The result follows directly from \eqref{eq:upsilonp1a}.
\end{proof}
\noindent
\textit{Proof (of Theorem \ref{interpolation}).}  Proofs for the cases $m=0$ and $m=1$ can be found in \cite{M2003} and \cite{MR2010}, respectively. 

\noindent
Consider the expansion
\begin{equation} \label{eq:uexpansion}
u(x)-W_{p,\Delta}^m(x)=u(x)-v(x)+v(x)-\pi^{\Delta}v(x)+\pi^{\Delta}v(x)-W_{p,\Delta}^m(x)
\end{equation}
\noindent
and address each pair in turn.  For the first pair from \eqref{eq:vdefinition} and Taylor's Theorem it follows that $\forall x \in (M-h/2,M+h/2) \, \exists \, {\hat x}$ s.t.
\begin{equation}
u(x)-v(x)=\frac{u^{(p+s+2)}(M)}{(p+s+2)!}({\hat x}-M)^{p+s+2}.
\end{equation}
Moving to the second pair, using the linearity of $\pi^{\Delta}$ and \eqref{eq:vdefinition}
\begin{equation}
v(x)-\pi^{\Delta}v(x)=\sum_{i=p+1}^{p+s+1}((x-M)^i-\pi^{\Delta}(x-M)^i)\frac{u^{(i)}(M)}{i!}\equiv\sum_{i=p+1}^{p+s+1}v_i(x)\frac{u^{(i)}(M)}{i!}.
\end{equation}
\noindent
Since the $v_i(x)$, $i=p+1, p+2, \ldots, p+s+1$, satisfy \eqref{eq:upsilonzerosa}-\eqref{eq:upsilonzerosb} it follows from Lemma \ref{upsilonp1lemma} that
\begin{equation}
v_i(x)=\Upsilon_{p+1,\Delta}^m(x)\Psi_{i-p-1}^m(x), \,
i=p+1, p+2, \ldots, p+s+1,
\end{equation}
\noindent
where $\Psi_{i-p-1}^m(x)$ is a polynomial of degree $i-p-1$.  The coefficients of $\Psi_i^m(x)$ can be derived following the same procedure for the case $m=0$ in Theorem 3.1 in \cite{M2003}.  Finally, since $\pi^{\Delta}(v(x)-W_{p,\Delta}^m(x)) \in S^{p,\Delta}$ from Lemma \ref{basis} it follows that
\begin{align}
&\pi^{\Delta}(v(x)-W_{p,\Delta}^m(x))=\sum_{i=2}^{p-2m}(v(\Gamma_i)-W_{p,\Delta}(\Gamma_i)) {\hat L}_i(x)\\
&+\sum_{i=1}^{m+1}(v^{(i)}(\Gamma_1)-W_{p,\Delta}^{(i)}(\Gamma_1)){\hat A}_i^{m,p}(x)+\sum_{i=1}^{m+1}(v^{(i)}(\Gamma_{p-2m+1})-W_{p,\Delta}^{(i)}(\Gamma_{p-2m+1})){\hat B}_i^{m,p}(x).
\end{align}
\noindent
\qed

\noindent
Remark: It is easily verified that $\forall m, ~\Psi_0^m(x)=1, ~\Psi_1^m(x)=x-M$.

\noindent
Remark: This theorem generalizes Theorem 3.1 of \cite{M2003} and Theorem 3 of \cite{MR2010}.

Let the $L^2$ and $H^k$ norms over $\Delta$ be denoted by $\Vert \cdot \Vert_{0, \Delta}$, $\Vert \cdot \Vert_{k,\Delta}, k=1, 2, \ldots m+1$, and their corresponding inner products by $(\cdot,\cdot)_{0,\Delta}$, $(\cdot,\cdot)_{k,\Delta}, \, k=1, 2, \ldots, m+1$, respectively. In addition, let the  $H^k$ seminorms over $\Delta$ be denoted by $\vert \cdot \vert_{k,\Delta}, k=1, 2, \ldots m+1$.  Finally let ${\tilde C}(\cdot)$ and ${\bar C}(\cdot)$ denote generic constants that depend on a set of parameters such as the size of the domain, $\vert \Delta \vert$, $p$, and the order of the derivative $k$.

\begin{theorem} \label{mainresult}
Let \eqref{eq:leftspline}-\eqref{eq:Brequirements} hold and let $u(x) \in C^{p+2}({\bar \Delta})$.  Then
\begin{align}
&\Vert u-W_{p,\Delta}^m \Vert_{k,\Delta}^2 \nonumber \\
&=\left( \frac{h}{2}\right)^{2p-2k+3}\left( \frac{u^{(p+1)}(M)}{(p-k+1)!}\right)^2 \sum_{i=0}^{m-k+1}({\hat J}_{p-k+1,i}^{m-k})^2(J_{-1,p-k+1-2i})^2\frac{2}{2p-2k-4i+3} \\
&+{\cal O}(h^{2p+5-2k}), \, k \le m+1. \label{eq:deltaerr} \\
\end{align}
\noindent
where the $J_{p-k,i}^{m-k}$ are the coordinates of ${\cal J}_p^m(\mu)$ in the basis ${\cal J}_q^{-1}(\mu)$, $q=0, 1, \ldots, p$.
\end{theorem}

To prove \eqref{eq:deltaerr} it is helpful to rewrite \eqref{eq:theorem1} in the case $s=0$ as
\begin{equation}
u(x)-W_{p,\Delta}^m(x)=\Upsilon_{p+1,\Delta}^m(x)\frac{u^{(p+1)}(M)}{(p+1)!}+\Pi_{p+2,\Delta}^m(x) \equiv \Pi_{p+1,\Delta}^m(x)+\Pi_{p+2,\Delta}^m(x), \label{eq:theorem1a}
\end{equation}
\noindent
where
\begin{align}
\Pi_{p+2,\Delta}^m(x,{\hat x})&=\frac{u^{(p+2)}(M)}{(p+2)!}({\hat x}-M)^{p+2}+\sum_{i=2}^{p-2m}(v(M_i)-u(M_i)) {\hat L}_{i,\Delta}^{m,p}(x) \nonumber \\
&+\sum_{l=0}^{m}(v^{(l)}(M_1)-u^{(l)}(M_1)){\hat A}_{l,\Delta}^{m,p}(x)+\sum_{l=0}^{m}(v^{(l)}(M_{p-2m+1})-u^{(l)}(M_{p-2m+1})){\hat B}_{l,\Delta}^{m,p}(x) \nonumber \\
&\equiv \sum_{i=1}^4\Pi_{p+2,\Delta}^{m,i}(x). \label{eq:xisum}
\end{align}
\noindent
Equations \eqref{eq:theorem1a}-\eqref{eq:xisum} imply that \textit{a priori} estimates of the error and its derivatives require bounds on the terms
\begin{equation}
\int_{\Delta} (\Pi_{p+1,\Delta}^m(x))^{(k)} (\Pi_{p+2,\Delta}^m(x))^{(k)}dx, \, \int_{\Delta} ((\Pi_{p+2,\Delta}^m(x))^{(k)})^2dx \label{eq:crosstermsa}
\end{equation}
\noindent
along with a formula for the leading term
\begin{equation}
\int_{\Delta} ((\Pi_{p+1,\Delta}^m(x))^{(k)})^2dx. \label{eq:leadingterm}
\end{equation}  
\noindent
Applying Cauchy-Schwarz to \eqref{eq:crosstermsa} yields
\begin{equation}
\int_{\Delta} (\Pi_{p+1,\Delta}^m(x))^{(k)} (\Pi_{p+2,\Delta}^m(x))^{(k)}dx \le \Vert \Pi_{p+1,\Delta}^m(x)\Vert_k\Vert \Pi_{p+2,\Delta}^m(x)\Vert_k. \label{eq:crosstermsb}
\end{equation}
\noindent
Cauchy-Schwarz also implies that
\begin{equation}
\int_{\Delta} ((\Pi_{p+2,\Delta}^m(x))^{(k)})^2dx \le 4 \sum_{i=1}^4 \Vert \Pi_{p+2,\Delta}^{m,i}\Vert_{k,\Delta}^2. \label{eq:crosstermsc}
\end{equation}

It is helpful to split the proof of the theorem into a series of lemmas beginning with bounds on the coefficients of ${\hat L}_{i,\Delta}^{p,m}(x)$, ${\hat A}_{i,\Delta}^{p,m}(x)$, and ${\hat B}_{i,\Delta}^{p,m}(x)$.  These bounds depend on the separation between the distinct roots of $\Upsilon_{p+1,\Delta}^m(x)$, $M_i$, $i=1, 2, \ldots, p-2m+1$, that is, the scaled roots of $\Upsilon_{p-2m-1,\Delta}^{-m-2}(\mu)$ along with the endpoints of $\Delta$.
\begin{lemma} \label{separation}
Let $\mu_1 < \mu_2 < \cdots < \mu_{p-2m-1}$ be the roots of $\Upsilon_{p-2m-1,\Delta}^{-m-2}(\mu)$.  Then
\begin{align}
\mu_2+1 &\ge \frac{4(m+2)}{R+2(m+2)}, \, 1-\mu_{p-2m} \ge \frac{4(m+2)}{R+2(m+2)}, \label{eq:spacinga} \\
\mu_{i+1}-\mu_i &\ge \frac{2^{11/4}\sqrt{m+2}}{\sqrt{R(R+4(m+2))}}, \label{eq:spacingb}
\end{align}
\noindent
where $R=2(p-2m-1)(p+2)$.
\end{lemma}
\begin{proof}
\cite{V}
\end{proof}
\noindent

\begin{lemma}
Equations \eqref{eq:spacinga}-\eqref{eq:spacingb} imply that $\exists \, {\tilde C}(p,m)$ s.t. that the roots of $\Upsilon_{p+1,\Delta}^m(x)$, $M_j, \, j=1, 2, \ldots p-2m+1$ satisfy
\begin{equation}
\vert M_i-M_j \vert \ge {\tilde C}(p,m)h, \label{eq:space}
\end{equation}
\noindent
for all $i,j=2, 3, \ldots, p-2m$, $i \ne j$.
\end{lemma}
\begin{proof}
The result follows directly from Lemma \ref{separation}.
\end{proof}
\noindent
The next step is to derive bounds on $({\hat L}_{i,\Delta}^{p,m}(x))^{(k)}$, ${\hat A}_{i,\Delta}^{p,m}(x)^{(k)}$, and ${\hat B}_{i,\Delta}^{p,m}(x)^{(k)}$.  These are then used to obtain bounds on the terms $({\hat A}_l^{m,p}(x))^{(k)}$ and $({\hat B}_l^{m,p}(x))^{(k)}$ and from there bounds on norms of \eqref{eq:crosstermsb}-\eqref{eq:crosstermsc}.
\begin{lemma}\label{forestlemma}
The coefficients $A_{j,l}$ and $B_{j,l}$ of ${\hat A}_l^{m,p}(x)$ and ${\hat B}_l^{m,p}(x)$ in \eqref{eq:ahatdefinition}-\eqref{eq:Brequirements}, respectively, satisfy
\begin{equation}
A_{j,l}={\cal O}(h^{-j}), B_{j,l}={\cal O}(h^{-j}), j=0, 1, \ldots m-l, \, l=0, \ldots, m. \label{eq:ajl}
\end{equation}
\end{lemma}
\begin{proof}
From \eqref{eq:achecka} and \eqref{eq:space} it follows that
\begin{equation}
({\hat G}_{p,\Delta}^m(M_1))^{(k)} = {\cal O}(h^{-k-1}). \label{eq:ghato}
\end{equation}
\noindent
From \eqref{eq:nonzero} 
\begin{equation}
(G_{p,\Delta}^m(M_1))^{(k)}={\cal O}(h^{-k}), \label{eq:Gderiv}
\end{equation} holds for $k=0$.  By induction assume it is true for $k=0, 1, \ldots, K$.  Using \eqref{eq:abarderiv1} together with \eqref{eq:ghato} and the induction hypothesis yields
\begin{equation}
\vert (G_{p,\Delta}^m(M_1))^{(K+1)} \vert \le {\tilde C}(p,m) \sum_{i=0}^{K}h^{-i}h^{-K-1+i}={\cal O}(h^{-K-1}). \label{eq:go}
\end{equation}
\noindent
Equation \eqref{eq:AJLa} implies that \eqref{eq:ajl} holds for $A_{j,l}$ with $j=0$.  By induction assume \eqref{eq:ajl} holds for $A_{j,l}$ with $j=0, 1, \ldots, J$.  From \eqref{eq:AJLc}, \eqref{eq:go}, and the induction hypothesis
\begin{equation}
\vert A_{J+1,l} \vert \le {\tilde C}(p,m)\sum_{i=0}^{J}h^{-i}h^{-J+1+i} = {\cal O}(h^{-J-1}),
\end{equation}
\noindent
thus establishing \eqref{eq:ajl} for the $A_{j,l}$.  The proof for the $B_{j,l}$ follows directly.
\end{proof}
\noindent
\begin{lemma}
Let the ${\hat L}_{i}^{m,p}(x)$, ${\hat A}_{i}^{m,p}(x)$, and ${\hat B}_{i}^{m,p}(x)$ be defined by \eqref{eq:lhatdefinition}-\eqref{eq:Brequirements}.  Then
\begin{align}
({\hat L}_{i}^{m,p}(x))^{(k)}&={\cal O}(h^{-k}), \, i=2, 3, \ldots, p-2m-1, \, k=0, 1, \ldots , m+1, \label{eq:laba} \\
({\hat A}_l^{m,p}(x))^{(k)}&={\cal O}(h^{-k+l}), \, l=0, 1, \ldots, m, \, k=0, 1, \ldots, m+1, \label{eq:labb} \\
({\hat B}_l^{m,p}(x))^{(k)}&={\cal O}(h^{-k+l}), \, l=0, 1, \ldots, m, \, k=0, 1, \ldots, m+1. \label{eq:labc}
\end{align}
\end{lemma}
\begin{proof}
Each term ${\hat L}_i^{m,p}(x)$ is of the form $\frac{x-\Gamma_j}{\Gamma_i-\Gamma_j}.$  Equation \eqref{eq:space} implies that
\begin{equation}
{\hat L}_i^{m,p}(x)={\cal O}(1). \label{eq:lhato}
\end{equation}  Taking the log of both sides of \eqref{eq:lhatdefinition} and then differentiating yields
\begin{equation} \label{eq:lhatprime}
{\hat L}_i^{m,p}(x)^{\prime}={\hat L}_i^{m,p}(x){\tilde L}_i^{m,p}(x),
\end{equation}
\noindent
where
\begin{equation}
{\tilde L}_i^{m,p}(x)=\left( \frac{m+1}{x-\Gamma_1}+\frac{m+1}{x-\Gamma_{p-2m+1}}+\sum_{j=2,j \ne i}^{p-2m}\frac{1}{x-\Gamma_j} \right).
\end{equation}
\noindent
Differentiating ${\tilde L}_i^{m,p}(x)$ $l$ times results in
\begin{equation}
({\tilde L}_i^{m,p}(x))^{(l)}=(-1)^ll! \left( \frac{m+1}{(x-\Gamma_1)^{l+1}}+\frac{m+1}{(x-\Gamma_{p-2m+1})^{l+1}}+\sum_{j=2,j \ne i}^{p-2m}\frac{1}{(x-\Gamma_j)^{l+1}} \right). \label{eq:ltildederiv}
\end{equation}
\noindent
Then \eqref{eq:space} and \eqref{eq:ltildederiv} imply that
\begin{equation}
({\tilde L}_i^{m,p}(x))^{(l)}={\cal O}(h^{-l-1}), \, l=0, 1, \ldots, m. \label{eq:ltildederiva}
\end{equation}
\noindent
Differentiating \eqref{eq:lhatprime} $l-1$ times yields
\begin{equation}
({\hat L}_i^{m,p}(x))^{(l)}=\sum_{j=0}^{l-1}\binom{l-1}{j}({\hat L}_i^{m,p})^{(j)}({\tilde L}_i^{m,p}(x))^{(l-1-j)}. \label{eq:lhatderiva}
\end{equation}
\noindent
Equation \eqref{eq:lhato} confirms that \eqref{eq:laba} holds for $l=0$.  By induction assume \eqref{eq:laba} holds for $l=0, 1, \ldots, L$.  Then the induction hypothesis with \eqref{eq:ltildederiva} and \eqref{eq:lhatderiva} implies that
\begin{equation}
\vert ({\hat L}_i^{m,p}(x))^{(L+1)}\vert \le {\tilde C}(p,m)\sum_{j=0}^Lh^{-j-1}h^{-L+j}={\cal O}(h^{-L-1}).
\end{equation}
\noindent
From \eqref{eq:ahatderivaa}
\begin{equation}
\vert ({\hat A}_l^{m,p}(x))^{(k)}\vert \le \frac{1}{l!}\sum_{i=0}^k\binom{k}{i}\vert (F_{l,\Delta}^m(x))^{(i)}\vert \times \vert (G_{p,\Delta}^m(x))^{(k-i)}\vert. \label{eq:ahatbound1}
\end{equation}
\noindent
Equation \eqref{eq:ahatderivab} implies that
\begin{align}
\vert (F_{l,\Delta}^m(x))^{(i)}\vert& \le \sum_{j={\rm max}(0,i-l)}^{m-l}\vert A_{j,l} \vert \frac{(j+1)!}{(j+l-i)!} \vert (x-M_1)^{j+l-i} \vert \le C(p,m)\sum_{j={\rm max}(0,i-l)}^{m-l}h^{-j}h^{j+l-i} \nonumber \\
&\le C(p,m)h^{l-i}, \label{eq:ahatbound2}
\end{align}
\noindent
using Lemma \ref{forestlemma} and Taylor's Theorem.  Taylor's Theorem and the first induction argument of Lemma \ref{forestlemma} yield
\begin{equation}
\vert (G_{p,\Delta}^m(x))^{(k-i)}\vert \le C(p,m)h^{-(k-i)}. \label{eq:ahatbound3}
\end{equation}
\noindent
Equations \eqref{eq:ahatbound1}-\eqref{eq:ahatbound3} give \eqref{eq:labb} while a comparable argument leads to \eqref{eq:labc}.
\end{proof}

\begin{lemma}
Assume that $u(x) \in C^{p+2}({\bar \Delta})$.  Then following the definitions in \eqref{eq:xisum} for any $0 \le k \le m+1,$ and $i=1, 2, 3, 4$, \, $\exists \, {\tilde C}(p,m,k)$ and ${\bar C}(p,m,k)$ s.t.
\begin{align} 
\Vert \Pi_{p+2,\Delta}^{m,i} \Vert_{k,\Delta}^2 &\le {\tilde C}(p,m,k) h^{2p-2k+5}\Vert u \Vert_{p+2}^2 ={\cal O}(h^{2p-2k+5}), \label{eq:lemmaboundsa} \\
\Vert \Pi_{p+2,\Delta}^{m} \Vert_{k,\Delta}^2 &\le {\bar C}(p,m,k) h^{2p-2k+5}\Vert u \Vert_{p+2}^2 ={\cal O}(h^{2p-2k+5}). \label{eq:lemmaboundsa1}
\end{align}
\end{lemma}
\begin{proof}
From the proof of Theorem \ref{interpolation} it follows that for any $0 \le j \le m+1$ and any $x \in \Delta$ $\exists \, {\hat x} \in \Delta$ s.t.
\begin{equation}
(\Pi_{p+2}^{m,1}(x))^{(j)}=\frac{u^{(p+2)}(M)}{(p+2-j)!}({\hat x}-M)^{p+2-j}. \label{eq:xi2errb}
\end{equation}
\noindent
Squaring \eqref{eq:xi2errb} and integrating over $\Delta$ yields 
\begin{align}
\vert \Pi_{p+2,\Delta}^{m,1} \vert_{j,\Delta}^2 &\le {\tilde C}(p,m,j) h^{2p-2j+5}\vert u \vert_{p+2}^2 \le {\tilde C}(p,m,j) h^{2p-2j+5}\Vert u \Vert_{p+2}^2={\cal O}(h^{2p-2j+5}), \label{eq:xi2errc}
\end{align}
\noindent
since $({\hat x}-M)^{p+2-j} \le {\tilde C}(p,m,j)h^{p+2-j}$ and $u(x) \in C^{p+2}({\bar \Delta})$.  Summing \eqref{eq:xi2errc} over $j$ from 0 to $k$ leads to \eqref{eq:lemmaboundsa} in the case $i=1$.  From \eqref{eq:vdefinition} and Taylor's Theorem
\begin{align}
\vert v(M_i)-u(M_i) \vert &\le {\tilde C}(p,m) h^{p+2}u^{(p+2)}(M_i), \, i=1, 2, \ldots, p-2m+1, \label{eq:taylora} \\
\vert v^{(l)}(M_i) -u^{(l)}(M_i) \vert&\le {\tilde C}(p,m,l)h^{p+2-l}u^{(p+2)}(M_i), \, i=1,p-2m+1, \, l=1, 2, \ldots, m, \label{eq:taylorb}
\end{align}
\noindent
or, using the Mean Value Theorem of Integrals,
\begin{align}
\vert v(M_i)-u(M_i) \vert &\le {\tilde C}(p,m) h^{p+2}\Vert u^{(p+2)}\Vert_{0,\Delta}= {\cal O}(h^{p+2}), \, i=1, 2, \ldots, p-2m+1, \label{eq:taylorc} \\
\vert v^{(l)}(M_i) -u^{(l)}(M_i) \vert&\le {\tilde C}(p,m,l)h^{p+2-l}\vert u^{(p+2)}(M_i)\vert_{l,\Delta}={\cal O}(h^{p+2-l}), \\
&i=1,p-2m+1, \, l=0, 1, \ldots, m. \label{eq:taylord}
\end{align}
\noindent
From \eqref{eq:xisum}
\begin{equation}
(\Pi_{p+2,\Delta}^{m,2}(x))^{(j)}\equiv\sum_{i=2}^{p-2m}(v(M_i)-u(M_i))({\hat L}_i^{m,p}(x))^{(j)}, \label{eq:pi22}
\end{equation}
\noindent
and thus using Cauchy-Schwarz,
\begin{align}
((\Pi_{p+2,\Delta}^{m,2}(x))^{(j)})^2&\le(p-2m-1)\sum_{i=2}^{p-2m}(v(M_i)-u(M_i))^2(({\hat L}_i^{m,p}(x))^{(j)})^2. \label{eq:xi3erra}
\end{align}
\noindent
Integrating over $\Delta$ and summing over $j=0, 1, \ldots, k,$ yields
\begin{align}
\Vert \Pi_{p+2,\Delta}^{m,2} \Vert_{k,\Delta}^2&\le \sum_{j=0}^k (p-2m-1)\sum_{i=2}^{p-2m}(v(M_i)-u(M_i))^2 \int_{\Delta}(({\hat L}_i^{m,p}(x))^{(j)})^2dx \nonumber \\
& \le {\tilde C}(p,m,k)\Vert u \Vert_{p+2}^2\sum_{j=0}^k h^{2p+4}hh^{-2j} \\
&\le {\tilde C}(p,m,k)h^{2p-2k+5}\Vert u \Vert_{p+2}^2={\cal O}(h^{2p+5-2k}),
\end{align}
\noindent\
where again Cauchy-Schwarz has been used together with \eqref{eq:laba} and \eqref{eq:taylorc}.  From \eqref{eq:xisum}
\begin{align}
(\Pi_{p+2,\Delta}^{m,3})^{(j)} &= \sum_{l=0}^{m}(v^{(l)}(M_1)-u^{(l)}(M_1))({\hat A}_l^{m,p}(x))^{(j)}.
 \label{eq:xi4erra}
\end{align}
\noindent
Applying Cauchy-Schwarz to \eqref{eq:xi4erra} as above leads to
\begin{align}
\Vert \Pi_{p+2,\Delta}^{m,3} \Vert_{k,\Delta}& \le (m+1) \sum_{j=0}^k\sum_{l=0}^m(v^{(l)}(M_1)-u^{(l)}(M_1))^2\int_{\Delta}(({\hat A}_l^{m,p}(x))^{(j)})^2dx \nonumber \\
&\le {\tilde C}(p,k,m)\Vert u \Vert_{p+2}^2 \sum_{j=0}^kh^{2p+4-2l}hh^{-2j+2l} \\
&\le {\tilde C}(p,m,k) h^{2p-2k+5}\Vert u \Vert_{p+2}^2= {\tilde C}(p,m,k)h^{2p+5-2k},
\end{align}
\noindent
where  \eqref{eq:labb} and \eqref{eq:taylord} have been used.  The same approach establishes \eqref{eq:lemmaboundsa} with $i=4$. Equation \eqref{eq:lemmaboundsa1} follows immediately.
\end{proof}

\begin{lemma}
Let
\begin{align}
{\hat J}_{p,i}^m&\equiv(-1)^i\binom{m+1
}{i}\frac{\prod_{j=0}^{2i-1}(p-j)}{\prod_{j=0}^{i-1}(2p-2j-2i+1)\prod_{j=0}^{i-1}(2p-2m-2i-1+2j)}, \, i=1, 2, \ldots m+1, \label{eq:jcoeff} \\
{\hat J}_{p,0}^m&=1, \label{eq:jcoeff0}
\end{align}
\noindent
where $\prod_{j=0}^{-1} \equiv 1$.  Then
\begin{equation} 
{\cal J}_{p}^{m}(\mu)=\sum_{i=0}^{m+1}{\hat J}_{p,i}^m{\cal J}_{p-2i}^{-1}(\mu). \label{eq:philegendrea}
\end{equation}
\end{lemma}
\begin{proof}
If $m=0$, \eqref{eq:philegendrea} simplifies to
\begin{align}
{\cal J}_p^0(\mu)&=\sum_{i=0}^1(-1)^i\binom{1}{i}\frac{\prod_{j=0}^{2i-1}(p-j)}{\prod_{j=0}^{i-1}(2p-2j-2i+1)\prod_{j=0}^{i-1}(2p-2i-1+2j)}{\cal J}_{p-2i}^{-1}(\mu) \nonumber \\
&={\cal J}_p^{-1}(\mu)-\frac{p(p-1)}{(2p-1)(2p-3)}{\cal J}_{p-2}^{-1}(\mu),
\end{align}
\noindent
which is true by \eqref{eq:phi4} with $m=0$. By induction assume \eqref{eq:jcoeff}-\eqref{eq:philegendrea} holds for $m=0, 1, \ldots, {\bar M}$.  Substituting ${\bar M}+1$ for $m$ in \eqref{eq:phi4} and using \eqref{eq:philegendrea} with ${\bar M}$ substituted for $m$ in both terms and $p-2$ in place of $p$ in the second term yields
\begin{align}
{\cal J}_p^{{\bar M}+1}(\mu)&=\sum_{i=0}^{{\bar M}+1}{\hat J}_{p,i}^{\bar M}{\cal J}_{p-2i}^{-1}(\mu)-\frac{p(p-1)}{(2p-2{\bar M}-3)(2p-2{\bar M}-5)}\sum_{i=0}^{{\bar M}+1}{\hat J}_{p-2,i}^{\bar M}{\cal J}_{p-2-2i}^{-1}(\mu). \label{eq:legendsumaa}
\end{align}
\noindent
Extracting the first term from the first sum of \eqref{eq:legendsumaa}, the last term from the second sum, then re-indexing the second sum and adding it to the first sum yields
\begin{align}
{\cal J}_p^{{\bar M}+1}(\mu)&={\hat J}_{p,0}^{\bar M}{\cal J}_p^{-1}(\mu)-\frac{p(p-1)}{(2p-2{\bar M}-3)(2p-2{\bar M}-5)}{\hat J}_{p-2,{\bar M}+1}^{\bar M}{\cal J}_{p-2{\bar M}-4}^{-1}(\mu) \nonumber \\
&+\sum_{i=1}^{{\bar M}+1}{\cal J}_{p-2i}^{-1}(\mu)\left( {\hat J}_{p,i}^{\bar M}-\frac{p(p-1)}{(2p-2{\bar M}-3)(2p-2{\bar M}-5)}{\hat J}_{p-2,i-1}^{\bar M}\right), \label{eq:legendsuma}
\end{align}
\noindent
The first term on the right in \eqref{eq:legendsuma} is the first term in \eqref{eq:philegendrea} with $m={\bar M}+1$ since ${\hat J}_{p,0}^{\bar M}={\hat J}_{p,0}^{{\bar M}+1}=1$.  Substituting \eqref{eq:jcoeff} into the second term in \eqref{eq:legendsuma} gives
\begin{align}
&-\frac{p(p-1)}{(2p-2{\bar M}-3)(2p-2{\bar M}-5)}{\hat J}_{p-2,{\bar M}+1}^{\bar M}{\cal J}_{p-2{\bar M}-4}^{-1}(\mu) \nonumber \\
&=(-1)^{{\bar M}+2}\frac{p(p-1)\prod_{j=0}^{2{\bar M}+1}(p-2-j)}{\prod_{j=0}^{\bar M}(2p-2j-2{\bar M}-5)(2p-2{\bar M}-3)(2p-2{\bar M}-5)\prod_{j=0}^{\bar M}(2p-4{\bar M}-7+2j)},
\end{align}
\noindent
which further simplifies to
\begin{align}
&=(-1)^{{\bar M}+2}\frac{\prod_{j=0}^{2{\bar M}+3}(p-j)}{\prod_{j=0}^{{\bar M}+1}(2p-2j-2{\bar M}-3)\prod_{j=0}^{{\bar M}+1}(2p-4{\bar M}-7+2j)}={\hat J}_{p,{\bar M}+2}^{{\bar M}+1}, \label{eq:legendsumab}
\end{align}
\noindent
by absorbing $p(p-1)$ into the numerator product and then re-indexing it and re-indexing both products in the denominator after absorbing into the first  product the term $2p-2{\bar M}-3$ and the term $2p-2{\bar M}-5$ into the second. Substituting \eqref{eq:jcoeff} into each term in the sum in parenthesis in \eqref{eq:legendsuma} and re-indexing the products in the numerator and denominator associated with ${\hat J}_{p-2,i-1}^{\bar M}$ results in
\begin{align}
&\left( {\hat J}_{p,i}^{\bar M}-\frac{p(p-1)}{(2p-2{\bar M}-3)(2p-2{\bar M}-5)}{\hat J}_{p-2,i-1}^{\bar M}\right) \nonumber \\
&=(-1)^i \frac{\binom{{\bar M}+1}{i}\prod_{j=0}^{2i-1}(p-j)}{\prod_{j=0}^{i-1}(2p-2j-2i+1)\prod_{j=0}^{i-1}(2p-2{\bar M}-2i-1+2j)} \nonumber \\
&+(-1)^i \frac{\binom{{\bar M}+1}{i-1}\prod_{j=0}^{2i-1}(p-j)}{\prod_{j=1}^{i-1}(2p-2j-2i+1)\prod_{j=-1}^{i-1}(2p-2{\bar M}-2i-1+2j)}, \label{eq:legendsumb}
\end{align}
\noindent
where both terms $2p-2{\bar M}-3$ and $2p-2{\bar M}-5$ have been absorbed into the second product in the second term of \eqref{eq:legendsumb}.  Equation \eqref{eq:legendsumb} can be further simplified as
\begin{align}
&=\frac{({\bar M}+2-i)(2p-2{\bar M}-2i-3)+i(2p-2i+1)}{i!({\bar M}+2)-i)!(2p-2{\bar M}-2i-3)} \nonumber \\
&\times(-1)^i\frac{({\bar M}+1)!\prod_{j=0}^{2i-1}(p-j)}{\prod_{j=0}^{i-1}(2p-2j-2i+1)\prod_{j=0}^{i-1}(2p-2{\bar M}-2i-1+2j)}, \label{eq:legendsumc}
\end{align}
\noindent
or
\begin{align}
&=(-1)^i\binom{{\bar M}+2}{i}\frac{\prod_{j=0}^{2i-1}(p-j)(2p-2{\bar M}-3)}{\prod_{j=0}^{i-1}(2p-2j-2i+1)\prod_{j=1}^i(2p-2{\bar M}-2i-3+2j)(2p-2{\bar M}-2i-3)}, \label{eq:legendsumd}
\end{align}
\noindent
 since $({\bar M}+2-i)(2p-2{\bar M}-2i-3)+i(2p-2i+1)=({\bar M}+2)(2p-2{\bar M}-3)$, 
Then noting that the last term in the second product in the denominator of \eqref{eq:legendsumd} with $j=i$ is $2p-2{\bar M}-3$ and that when $j=0$, $2p-2{\bar M}-2i-3+2j=2p-2{\bar M}-2i-3$, \eqref{eq:legendsumd} further simplifies to
\begin{align}
&=(-1)^i\binom{{\bar M}+2}{i}\frac{\prod_{j=0}^{2i-1}(p-j)}{\prod_{j=0}^{i-1}(2p-2j-2i+1)\prod_{j=0}^{i-1}(2p-2{\bar M}-2i-3+2j)} ={\hat J}_{p,i}^{{\bar M}+1},
\end{align}
\noindent
which together with \eqref{eq:legendsuma} and \eqref{eq:legendsumab} completes the proof.
\end{proof}

\noindent
Remark: Equations \eqref{eq:jcoeff}-\eqref{eq:philegendrea} can be used to calculate the values of ${\cal J}_p^m(\mu)$ at quadrature points.

\begin{lemma}
The $k$-seminorm of $\Upsilon_{p+1,\Delta}^m(x)$ satisfies
\begin{align}
&\int_{M-h/2}^{M+h/2}(\Upsilon_{p+1,\Delta}^m(\mu))^{(k)})^2d\mu \nonumber \\
&=\left( \frac{h}{2}\right)^{2p-2k+3} \left( \frac{p!}{(p-k)!}\right)^2 \sum_{i=0}^{m-k+1}({\hat J}_{p-k,i}^{m-k})^2(J_{-1,p-k-2i})^2\frac{2}{2(p-k-2i)+1}={\cal O}(h^{2p-2k+3}). \label{eq:Upsilonerr}
\end{align}
\end{lemma}

\begin{proof}
Taking $k$ derivatives of \eqref{eq:defUpsilon} with $i=p+1$ using \eqref{eq:jderiv} yields
\begin{equation}
(\Upsilon_{p+1,M}^m(x))^{(k)}=\left( \frac{h}{2}\right)^{p+1-k}({\cal J}_{p+1}^m(2(x-M)/h))^{(k)}. \label{eq:Upsilona}
\end{equation}
\noindent
Squaring \eqref{eq:Upsilona} and integrating over $\Delta$ results in
\begin{equation}
\int_{M-h/2}^{M+h/2} ((\Upsilon_{p+1,M}^m(x))^{(k)})^2dx=\left( \frac{h}{2}\right)^{2p-2k+3} \int_{-1}^1(({\cal J}_{p+1}^m(\mu))^{(k)})^2d\mu. \label{eq:Upsilonb}
\end{equation}
\noindent
To evaluate the right side of \eqref{eq:Upsilonb} first differentiate \eqref{eq:philegendrea} $k$ times and use \eqref{eq:jderiv} to obtain
\begin{equation}
({\cal J}_{p+1}^m(\mu))^{(k)}=(p+1)p \cdots (p-k+2){\cal J}_{p-k+2}^{m-k} (\mu) \nonumber=\frac{(p+1)!}{(p-k+1)!}\sum_{i=0}^{m-k+1}{\hat J}_{p-k+1,i}^{m-k}{\cal J}_{p-k+1-2i}^{-1}(\mu). \label{eq:Upsilonc}
\end{equation}
\noindent
and thus from \eqref{eq:JNormalize} with $m=-1$,
\begin{align}
\int_{M-h/2}^{M+h/2}& ((\Upsilon_{p+1,M}^m(x))^{(k)})^2dx=\left( \frac{h}{2}\right)^{2p-2k+3} \frac{(p+1)!}{(p-k+1)!}\int_{-1}^1\left(\sum_{i=0}^{m-k+1}{\hat J}_{p-k+1,i}^{m-k}{\cal J}_{p-k+1-2i}^{-1}(\mu)\right)^2d\mu \nonumber \\
&=\left( \frac{h}{2}\right)^{2p-2k+3} \frac{(p+1)!}{(p-k+1)!}\sum_{i=0}^{m-k+1}({\hat J}_{p-k+1,i}^{m-k})^2(J_{-1,p-k+1-2i})^2\frac{2}{2(p-k+1-2i)+1},
\end{align}
\noindent
since the Legendre polynomials are orthogonal and $\int_{-1}^1 (P_p(\mu))^2d\mu=\frac{2}{2p+1}$.
\end{proof}

\begin{lemma}
Let $u(x) \in C^{p+2}({\bar \Delta})$.  Then
\begin{align}
\Vert \Pi_{p+1,\Delta}^m \Vert_{k,\Delta}^2&=\left(\frac{u^{p+1}(M)}{(p-k+1)!}\right)^2 \left( \frac{h}{2}\right)^{2p-2k+3}  \sum_{i=0}^{m-k+1}\frac{2({\hat J}_{p-k+1,i}^{m-k})^2(J_{-1,p-k+1-2i})^2}{2p-2k-4i+3} \nonumber \\
&+O(h^{2p-2k+5})={\cal O}(h^{2p-2k+3}). \label{eq:Pi1err}
\end{align}
\end{lemma}
\begin{proof}
From \eqref{eq:theorem1a} and \eqref{eq:Upsilonerr} it follows for $j \le k$ that
\begin{align}
\vert \Pi_{p+1,\Delta}^m \vert_{j,\Delta}^2&=\left(\frac{u^{p+1}(M)}{(p-j+1)!}\right)^2 \left( \frac{h}{2}\right)^{2p-2j+3} \sum_{i=0}^{m-j+1}\frac{2({\hat J}_{p-j+1,i}^{m-j})^2(J_{-1,p-j+1-2i})^2}{2p-2j-4i+3}.
\end{align}
Since
\begin{equation}
\Vert \Pi_{p+1,\Delta}^m \Vert_{k,\Delta}^2=\sum_{j=0}^k \vert \Pi_{p+1,\Delta}^m \vert_{j,\Delta}^2,
\end{equation}
\noindent
the result follows.
\end{proof}
\noindent

\textit{Proof (of Theorem \ref{mainresult}).}
From \eqref{eq:theorem1a}
\begin{align}
&(u-W_{p,\Delta}^m,u-W_{p,\Delta}^m)_{k,\Delta}^2=(\Pi_{p+1,\Delta}^m+\Pi_{p+2,\Delta}^m,\Pi_{p+1,\Delta}^m+\Pi_{p+1,\Delta}^m)_{k,\Delta} \nonumber \\
&=(\Pi_{p+1,\Delta}^m,\Pi_{p+1,\Delta}^m)_{k,\Delta}+(\Pi_{p+2,\Delta}^m,\Pi_{p+2,\Delta}^m)_{k,\Delta}+2(\Pi_{p+1,\Delta}^m,\Pi_{p+2,\Delta}^m)_{k,\Delta} \label{eq:theorem2}
\end{align}
\noindent
Bounds on the first two terms on the right-hand side of the second step of \eqref{eq:theorem2} come from \eqref{eq:Pi1err} and \eqref{eq:lemmaboundsa1}, respectively.  The bound on the third term follows from \eqref{eq:Pi1err} and \eqref{eq:lemmaboundsa1} together with Cauchy-Schwarz.
\qed

\noindent
Remark: This theorem generalizes Corollary 2.6 of \cite{M2001} and Corollary 4 of \cite{MR2010}.

From the lemmas and Theorem \ref{interpolation}, point-wise superconvergence can now be demonstrated.
\begin{corollary}
Let $u(x) \in C^{p+3}({\bar \Delta})$ and \eqref{eq:defpsi}-\eqref{eq:Brequirements} hold.  For each $k$, $k=1, \ldots, m+1$, let $x_{j,k}$,  $j=1, \ldots, p+1-k$, be the roots of  $(\Upsilon_{p+1,\Delta}^m(x))^{(k)}$.  Then it follows that
\begin{equation}
\vert (u(x)-W_{p,\Delta}^m(x))^{(k)} \vert_{x=x_{j,k}} \le Ch^{p+2-k}. \label{eq:nodalsuperconvergencea}
\end{equation}
\end{corollary}
\begin{proof}
Equation \eqref{eq:nodalsuperconvergencea} follows from \eqref{eq:theorem1a}, \eqref{eq:laba}-\eqref{eq:labc}, \eqref{eq:xi2errb}, \eqref{eq:taylora}-\eqref{eq:taylorb}, \eqref{eq:pi22}, and \eqref{eq:xi4erra}.
\end{proof}
\noindent
Thus, the error of the $k^{th}$ derivative of the Hermite-Jacobi interpolant at the roots of $(\Upsilon_{p+1,\Delta}^m(x))^{(k)}$ is one order higher than the expected pointwise error and thus exhibits pointwise superconvergence. This implies the counter-intuitive result that the number of superconvergence points increases by one for every increase in the order of the derivative.

In the case $k=m+1$ the interpolant error simplifies considerably.
\begin{corollary}
Let \eqref{eq:leftspline}-\eqref{eq:Brequirements} hold and let $u(x) \in C^{p+2}({\bar \Delta})$.  Then if $k=m+1$,
\begin{equation}
\Vert u-W_{p,\Delta}^m \Vert_{m+1,\Delta}^2=\frac{h^{2p-2m+1}}{2p-2m+1}(u^{(p+1)}(M))^2\left(\frac{(p-m)!}{(2p-2m)!}\right)^2+{\cal O}(h^{2p-2m+3}). \label{eq:m1error}
\end{equation}
\end{corollary}
\begin{proof}
If $k=m+1$ the sum in \eqref{eq:deltaerr} reduces to one term with $i=0$.  Since ${\hat J}_{p-m-1,0}^{-1}=1$ and $J_{-1,p-m-1}=\frac{2^{p-m-1}(p-m-1)!(p-m)!}{(2p-2m-1)!}$ the result follows by direct calculation.
\end{proof}

\noindent
Remark: This corollary generalizes Theorem 4.2 of \cite{M2001} and Corollary 5 of \cite{MR2010}.

Consider now the domain $\Omega$ and partition it into $N$ intervals $\Omega_j=(x_{j-1},x_j), j=1, 2, \ldots, N$, letting
\begin{align}
\Gamma_{\Omega} & \equiv \{ a=x_0 < x_1 < \cdots < x_N=b \}.
\end{align}
\noindent
Furthermore, let $h_j=x_j-x_{j-1}$ and $m_j=(x_{j-1}+x_j)/2$, be the interval (element) lengths and midpoints of $\Omega_j, j=1, 2, \ldots , N$, respectively, with $H={\rm max}_j h_j$, $\beta H={\rm min}_jh_j$ where $\beta$ is bounded from below independently of $N$. Let $S^{p,m,\Gamma_{\Omega}}$ be the space of piecewise $C^{m+1}(\Gamma_{\Omega})$ polynomials of degree $p \ge 2m+1$, $j=1, 2, \ldots, N$.  Consider $W_{p,m,\Gamma_{\Omega}} \in S^{p,m,\Gamma_{\Omega}}$
\begin{align}
W_{p,m,\Gamma_{\Omega}}&=\begin{cases}
\sum_{j=1}^{N+1}\sum_{l=0}^mW_{l,j}\Phi_{l,j}^m(x) & p=2m+1, \\
\sum_{j=1}^{N+1}\sum_{l=0}^mW_{l,j}\Phi_{l,j}^m(x)+\sum_{j=1}^N\sum_{i=2m+2}^p W_{i,j}\Upsilon_{i,j}^m(x) & p \ge 2m+2,
\end{cases} \label{eq:HLOMEGAa}
\end{align}
\noindent
where 
\begin{equation}
\Phi_{l,j}^m(x)=\begin{cases}
\Phi_{l,L,m_j}^m(x) & x \in \Omega_j, \\
\Phi_{l,R,m_{j-1}}^m(x) & x \in \Omega_{j-1}, \\
0 & {\rm otherwise},
\end{cases} \label{eq:HLOMEGAb}
\end{equation} 
\noindent
and
\begin{equation}
\Upsilon_{i,j}^m(x)=\begin{cases}
\Upsilon_{i,\Omega_j}^m(x) & x \in \Omega_j, \\
0 & {\rm otherwise}.
\end{cases} \label{eq:HLOMEGAc}
\end{equation}
\noindent
The $W_{i,j}$ are determined by applying \eqref{eq:left}-\eqref{eq:middle} to each element $\Omega_j, j=1, 2, \ldots , N$.  

\begin{theorem} \label{maintheorem2}
If $u(x) \in C^{p+2}({\bar {\Omega}})$
\begin{align}
&\Vert u-W_{p,m,\Gamma_{\Omega}} \Vert_{k,\Omega}^2 \nonumber \\
&=\sum_{j=1}^N\left( \frac{h_j}{2}\right)^{2p-2k+2}\left( \frac{u^{(p+1)}(m_j)}{(p+1)!}\right)^2 \left( \frac{p!}{(p-k)!}\right)^2 \sum_{i=0}^{m-k+1}({\hat J}_{p-k,i}^{m-k})^2(J_{-1,p-k-2i})^2\frac{2}{2(p-k-2i)+1} \nonumber \\
&+{\cal O}(H^{2p+3-2k}) \\
&\le{\tilde C}(p,m)H^{2p-2k+2}\Vert u \Vert_{p+1}^2+{\cal O}(H^{2p+3-2k}), k \le m+1. \label{eq:deltaerr1} \\
\end{align}
\end{theorem}
\begin{proof}
The equality follows by applying \eqref{eq:theorem2} to each interval $\Omega_j$, $j=1, 2, \ldots, N$, the fact that the bound on $\beta$ is independent of $N$ and $N$ is proportional to $1/H$.  The boundedness of $u^{p+1}(x)$ on $\Omega$ and the Mean Value of Integrals establishes the inequality.
\end{proof}

\noindent
Remark: This theorem generalizes Corollary 5 of \cite{MR2010}.

\section{Asymptotic Equivalence}

Let $u(x) \in C_0^{p+2}({\bar \Omega})$ be the solution of the linear equation
\begin{align}
f(x)&=\sum_{k=1}^{m+1}((-1)^kp_k(x)u^{(k)}(x))^{(k)}+p_0(x)u(x), \, x \in \Omega, \label{eq:classicalsolna} \\
&u^{(k)}(a)=u^{(k)}(b)=0, \, k=0, 1, \ldots, m, \label{eq:classicalsolnc}
\end{align}
\noindent
where $f(x), p_k(x) \in C^p({\bar \Omega}),  \, k=0, 1, \ldots, m+1$ and $p_{m+1}(x) > 0$ with $P_{\rm max}=\max_{k \in \{0, 1, \ldots, m+1\},x \in \Omega} |p_k(x) |$ and $p_{\rm max}^{\prime}=\max_{x \in \Omega}| p_{m+1}^{\prime}(x) |$.  Then $u(x)$ satisfies the weak form
\begin{align}
(f,v)&=\sum_{k=0}^{m+1}\sum_{j=1}^N(p_k(x)u(x),v(x))_{k,\Omega_j} \equiv  \sum_{k=0}^{m+1}a_k(u,v) \equiv a(u,v), \, \forall v \in H_0^{p}(\Omega) \label{eq:solna}
\end{align}
\noindent
together with \eqref{eq:classicalsolnc}.  The finite element approximation to $u(x)$, $U_{p,m,\Gamma_{\Omega}}(x) \in S_0^{p,m,\Gamma_{\Omega}}$, satisfies
\begin{align}
(f,V_{p,m,\Gamma_{\Delta}})&=\sum_{k=0}^{m+1}a_k(U_{p,m,\Gamma_{\Omega}},V_{p,m,\Gamma_{\Omega}})=a(U_{p,m,\Gamma_{\Omega}},V_{p,m,\Gamma_{\Omega}}), 
\, \forall V_{p,m,\Gamma_{\Omega}} \in S_0^{p,m,\Gamma_{\Omega}}. \label{eq:UFEMa}
\end{align}

First it will be necessary to derive an extension of the Poincar\'e inequality.
\begin{lemma} \label{poincare}
Let $g \in C_0^{m+1}({\bar \Omega})$ with $g^{(k)}(a)=g^{(k)}(b)=0$, $k=0, 1, \ldots m$.  Then $\exists \, {\tilde C}(k,\vert \Omega \vert)$ s.t.
\begin{equation}
{\tilde C}(k,\vert\Omega\vert)\Vert g \Vert_{k} \le \vert g \vert_{k}, \, k=1, 2, \ldots m+1. \label{eq:semibound}
\end{equation}
\end{lemma}
\begin{proof}
From the Fundamental Theorem of Calculus
\begin{equation}
g(x)=\int_a^xg^{\prime}(z)dz, \label{eq:ftc}
\end{equation}
\noindent
since $g(a)=0$.  Applying Cauchy-Schwarz to \eqref{eq:ftc} yields
\begin{equation}
| g(x) | \le \left( \int_a^x(g^{\prime}(z))^2\right)^{1/2}\sqrt{b-a}. \label{eq:ftca}
\end{equation}
\noindent
Squaring \eqref{eq:ftca} and integrating over $[a,b]$ results in
\begin{equation}
\int_a^bg^2(x)dx \le (b-a)^2\int_a^b(g^{\prime}(z))^2dz,
\end{equation}
\noindent
that is
\begin{equation}
\Vert g \Vert_0^2 \le {\hat C}(\vert \Omega \vert)\vert g \vert_1^2,
\end{equation}
\noindent
where ${\hat C}(\vert \Omega \vert)=(b-a)^2$.  Consequently
\begin{equation}
\Vert g \Vert_1^2=\Vert g \Vert_0^2+\vert g \vert_1^2 \le ({\hat C}(\vert \Omega \vert)+1)\vert g \vert_1^2, \label{eq:ftcf}
\end{equation}
\noindent
so that \eqref{eq:semibound} holds for $k=1$ with ${\tilde C}(\vert \Omega \vert)=\sqrt{1/(1+{\hat C}(\vert \Omega \vert))}$.  By induction assume \eqref{eq:semibound} holds for $k=1, 2, \ldots , K$.  Using the sequence of steps \eqref{eq:ftc}-\eqref{eq:ftcf} with $g^{(K)}(x)$ in place of $g(x)$ gives
\begin{equation}
\vert g \vert_{K}^2 \le {\bar C}(\vert \Omega \vert) \vert g \vert_{K+1}^2. \label{eq:ftcd}
\end{equation}
\noindent
Equation \eqref{eq:ftcd} together with the induction hypothesis implies that
\begin{equation}
({\tilde C}(K,\vert\Omega\vert))^2 \Vert g \Vert_{K}^2 \le \vert g \vert_{K}^2 \le {\bar C}(\vert \Omega \vert)\vert g \vert_{K+1}^2,
\end{equation}
\noindent
or
\begin{equation}
\Vert g \Vert_{K+1}^2 \le \left(1+\frac{{\bar C}(\vert \Omega \vert)}{{\tilde C}(K,\vert \Omega \vert)^2}\right)\vert g \vert_{K+1}^2,
\end{equation}
yielding \eqref{eq:semibound} with ${\tilde C}(K+1,\vert \Omega \vert)=\sqrt{1/(1+{\bar C}(\vert \Omega \vert)/{\tilde C}(K,\vert \Omega \vert)^2)}$.
\end{proof}
\begin{lemma}
Let $u,v \in H_0^{m+1}(\Omega)$.  Then $ \exists \, C(P_{\rm max},m)$ and ${\tilde C}(m,\vert \Omega \vert)$ s.t.
\begin{align}
\vert a(u,v) \vert& \le C(P_{\rm max},m)\Vert u \Vert_{m+1} \Vert v \Vert_{m+1}, \label{eq:energy1} \\
\Vert u \Vert_{m+1}^2& \le {\tilde C}(m,\vert \Omega \vert) a(u,u). \label{eq:energy2}
\end{align}
\end{lemma}
\begin{proof}
\begin{align}
| a(u,v) |& \le \sum_{k=0}^{m+1}| a_k(u,v) | \le \sum_{k=0}^{m+1}\max_{x \in \Omega}|p_k(x) | |u|_{k}|v|_{k} \nonumber \\
&\le P_{\rm max}\sum_{k=0}^{m+1}\| u \|_{m+1}\| v \|_{m+1} \le C(P_{\rm max},m)\| u \|_{m+1}\| v \|_{m+1}.
\end{align}
Equation \eqref{eq:energy2} follows from Lemma \ref{poincare}.
\end{proof}

\begin{theorem}
If $u(x) \in C_0^{p+2}(\Omega)$ satisfy \eqref{eq:classicalsolnc}-\eqref{eq:solna}
and let $U_{p,m,\Gamma_{\Omega} }\in S_0^{p,m,\Gamma_{\Omega}}$ satisfy \eqref{eq:UFEMa}.  Then $\exists \,{\tilde C}(p,m,P_{\rm max}, \vert \Omega \vert)$ s.t.
\begin{equation}
\Vert u-U_{p,m,\Gamma_{\Omega}} \Vert_{m+1} \le {\tilde C}(p,m,P_{\rm max},\vert \Omega \vert) H^{p-m}\Vert u \Vert_{p+1}. \label{eq:femerr}
\end{equation}
\end{theorem}

\begin{proof}
From \eqref{eq:UFEMa} it follows that
\begin{align}
a(u-U_{p,m,\Gamma_{\Omega}},u-U_{p,m,\Gamma_{\Omega}})&=a(u-U_{p,m,\Gamma_{\Omega}},u-W_{p,m,\Gamma_{\Omega}}+W_{p,m,\Gamma_{\Omega}}-U_{p,m,\Gamma_{\Omega}}) \nonumber \\
&=a(u-U_{p,m,\Gamma_{\Omega}},u-W_{p,m,\Gamma_{\Omega}}). \label{eq:hmp1error}
\end{align}
\noindent
Using \eqref{eq:hmp1error} together with Theorem \ref{maintheorem2} and \eqref{eq:energy2} results in
\begin{align}
C(m,\vert \Omega \vert)&\Vert u-U_{p,m,\Gamma_{\Omega}} \Vert_{m+1}^2 \le a(u-U_{p,m,\Gamma_{\Omega}},u-U_{p,m,\Gamma_{\Omega}}) \nonumber \\
&\le {\tilde C}(p,m,P_{\rm max})H^{p-m}\Vert u \Vert_{p+1}\Vert u-U_{p,m,\Gamma_{\Omega}} \Vert_{m+1}.
\end{align}
\end{proof}

\begin{theorem}
Let $u(x) \in C_0^{p+2}(\Omega)$ satisfy \eqref{eq:solna}
together with \eqref{eq:classicalsolnc}, let $W_{p,m,\Gamma_{\Omega}} \in S_0^{p,m,\Gamma_{\Omega}}$ be the Hermite-Jacobi interpolant \eqref{eq:HLOMEGAa}-\eqref{eq:HLOMEGAc} and let $U_{p,m,\Gamma_{\Omega}} \in S_0^{p,m,\Gamma_{\Omega}}$ be the finite element solution of \eqref{eq:UFEMa}.  Then $\exists \, {\tilde C}(p,m,P_{\rm max},p_{\rm max}^{\prime},\vert \Omega \vert)$ s.t.
\begin{equation}
\Vert W_{p,m,\Gamma_{\Omega}}-U_{p,m,\Gamma_{\Omega}} \Vert_{m+1,\Omega} \le {\tilde C}(p,m,P_{\rm max},p_{\rm max}^{\prime},\vert \Omega \vert)\Vert u \Vert_{p+2}H^{p-m+1}. \label{eq:superconvergence}
\end{equation}
\end{theorem}

\begin{proof}
\begin{align}
&a(W_{p,m,\Gamma_{\Omega}}-U_{p,m,\Gamma_{\Omega}},W_{p,m,\Gamma_{\Omega}}-U_{p,m,\Gamma_{\Omega}})=
\sum_{j=1}^Na(W_{p,m,\Gamma_{\Omega}}-U_{p,m,\Gamma_{\Omega}},W_{p,m,\Gamma_{\Omega}}-u)_{\Omega_j} \nonumber \\
&+a(W_{p,m,\Gamma_{\Omega}}-U_{p,m,\Gamma_{\Omega}},u-U_{p,m,\Gamma_{\Omega}})=
\sum_{j=1}^Na(W_{p,m,\Gamma_{\Omega}}-U_{p,m,\Gamma_{\Omega}},\Pi_{p+1,j}^m)_{\Omega_j} \nonumber \\
&+\sum_{j=1}^Na(W_{p,m,\Gamma_{\Omega}}-U_{p,m,\Gamma_{\Omega}},\Pi_{p+2,j}^m)_{\Omega_j}, \label{eq:super1}
\end{align}
\noindent
using \eqref{eq:theorem1a}, \eqref{eq:solna}-\eqref{eq:UFEMa}.  The second term in final step in \eqref{eq:super1} satisfies
\begin{equation}
\sum_{j=1}^Na(W_{p,m,\Gamma_{\Omega}}-U_{p,m,\Gamma_{\Omega}},\Pi_{p+2,j}^m)_{\Omega_j}=a(W_{p,m,\Gamma_{\Omega}}-U_{p,m,\Gamma_{\Omega}},\sum_{j=1}^N \Pi_{p+2,j}^m), \label{eq:super1a}
\end{equation}
\noindent
where $\Pi_{p+2,j}^m(x)$ has been extended to $\Omega \setminus \Omega_j$ by defining it to be zero outside of $\Omega_j$.  Then using \eqref{eq:lemmaboundsa1} with $k=m+1$ and \eqref{eq:energy1}, \eqref{eq:super1a} becomes
\begin{align}
\sum_{j=1}^Na(W_{p,m,\Gamma_{\Omega}}-U_{p,m,\Gamma_{\Omega}},\Pi_{p+2,j}^m)_{\Omega_j}& \le C(p,m,P_{\rm max})\Vert  W_{p,m,\Gamma_{\Omega}}-U_{p,m,\Gamma_{\Omega}}\Vert_{m+1}\sum_{j=1}^N \Vert \Pi_{p+2,j}^m \Vert_{m+1} \nonumber \\
& \le C(p,m,P_{\rm max},\vert \Omega \vert)H^{p-M+1}. \label{eq:secondterm}
\end{align}
\noindent
The first term in the final step in \eqref{eq:super1} satisfies
\begin{align}
\sum_{j=1}^Na&(W_{p,m,\Gamma_{\Omega}}-U_{p,m,\Gamma_{\Omega}},\Pi_{p+1,j}^m)_{\Omega_j}=\sum_{j=1}^Na_{m+1}(W_{p,m,\Gamma_{\Omega}}-U_{p,m,\Gamma_{\Omega}},\Pi_{p+1,j}^m)_{\Omega_j} \nonumber \\
&+\sum_{k=0}^m\sum_{j=1}^Na_{k}(W_{p,m,\Gamma_{\Omega}}-U_{p,m,\Gamma_{\Omega}},\Pi_{p+1,j}^m)_{\Omega_j}. \label{eq:firstterma}
\end{align}
\noindent
Applying Taylor's Theorem to the first term on the right of \eqref{eq:firstterma} leads to
\begin{align}
\sum_{j=1}^N&a_{m+1}(W_{p,m,\Gamma_{\Omega}}-U_{p,m,\Gamma_{\Omega}},\Pi_{p+1,j}^m)_{\Omega_j}=\sum_{j=1}^Np_{m+1}(M_j) (W_{p,m,\Gamma_{\Omega}}-U_{p,m,\Gamma_{\Omega}},\Pi_{p+1,j}^m)_{m+1,\Omega_j} \nonumber \\
&+\sum_{j=1}^Np^{\prime}(\chi_{j})((W_{p,m,\Gamma_{\Omega}}-U_{p,m,\Gamma_{\Omega}})^{(m+1)},(x-M_j)(\Pi_{p+1,j}^m)^{(m+1)})_{0,\Omega_j}. \label{firsttermb1}
\end{align}
\noindent
Due to the orthogonality of the Legendre polynomials the first term on the right of \eqref{firsttermb1} is zero.   The second term on the right of \eqref{firsttermb1} can be bounded as
\begin{align}
\sum_{j=1}^N&p^{\prime}(\chi_{m+1,j})((W_{p,m,\Gamma_{\Omega}}-U_{p,m,\Gamma_{\Omega}})^{(m+1)},(x-M_j)(\Pi_{p+1,j}^m)^{(m+1)})_{0,\Omega_j} \nonumber \\
&\le p_{\rm max}^{\prime}\sum_{j=1}^N\vert W_{p,m,\Gamma_{\Omega}}-U_{p,m,\Gamma_{\Omega}}\vert_{m+1,\Omega_j} \vert (x-M_j)(\Pi_{p+1,j}^m)^{(m+1)} \vert_{m+1,\Omega_j} \nonumber \\
&\le C(P_{\rm max},p_{\rm max}^\prime)\Vert W_{p,m,\Gamma_{\Omega}}-U_{p,m,\Gamma_{\Omega}}\Vert_{m+1} \sum_{j=1}^N\Bigg| \int_{\Omega_j}(x-M_j)(\Pi_{p+1,j}^m)^{(m+1)}dx \Bigg| \nonumber \\
&\le C(p,m,P_{\rm max},p_{\rm max}^\prime,\vert \Omega \vert)\Vert W_{p,m,\Gamma_{\Omega}}-U_{p,m,\Gamma_{\Omega}}\Vert_{m+1}H^{p-m+1}. \label{eq:firsttermb11}
\end{align}
\noindent
where the Cauchy-Schwarz inequality has been used, \eqref{eq:Upsilona}, and \eqref{eq:theorem1a} .  For $k < m+1$
\begin{align}
\sum_{j=1}^N&a_k(W_{p,m,\Gamma_{\Omega}}-U_{p,m,\Gamma_{\Omega}},\Pi_{p+1,j}^m)_{\Omega_j} \le \sum_{j=1}^N \vert W_{p,m,\Gamma_{\Omega}}-U_{p,m,\Gamma_{\Omega}}\vert_{k,\Omega_j}\Bigg|p(\chi_{k,j})\int_{\Omega_j}(\Pi_{p+1,j}^m)^{(k)}dx\Bigg| \nonumber \\
&\le C(p,m,P_{\rm max},\vert \Omega \vert)\Vert W_{p,m,\Gamma_{\Omega}}-U_{p,m,\Gamma_{\Omega}}\vert_{m+1}H^{p+1-k}, \label{eq:firsttermb12}
\end{align}
\noindent
where again \eqref{eq:Upsilona} has been used.  Combining \eqref{eq:super1}, \eqref{eq:secondterm}-\eqref{eq:firstterma}, \eqref{eq:firsttermb11}-\eqref{eq:firsttermb12} yields
\eqref{eq:superconvergence}.
\end{proof}

Thus, $\Vert W_{p,m,\Gamma_{\Omega}}-U_{p,m,\Gamma_{\Omega}} \Vert_{m+1}$  is one order higher than $\Vert u-U_{p,m,\Gamma_{\Omega}}\Vert_{m+1}$ and $\Vert u-W_{p,m,\Gamma_{\Omega}}\Vert_{m+1}$.  This superconvergence property implies the asymptotic equivalence of the finite element solution and the Hermite-Jacobi interpolant.
\begin{corollary}
Let $u(x) \in C_0^{p+2}(\Omega)$ be the solution of \eqref{eq:classicalsolnc}-\eqref{eq:solna}, $U_{p,m,\Gamma_{\Omega}} \in S_0^{p,m,\Gamma_{\Omega}}$ be the solution of \eqref{eq:UFEMa}, and let $W_{p,m,\Gamma_{\Omega}} \in S_0^{p,m,\Gamma_{\Omega}}$ satisfy \eqref{eq:HLOMEGAa}-\eqref{eq:HLOMEGAc}.  Then if 
\begin{equation}
\Vert u-W_{p,m,\Gamma_{\Omega}} \Vert_{m+1} \ge C(p,m)H^{p-m}, \label{eq:bound} 
\end{equation}
\begin{equation}
\frac{\Vert u-U_{p,m,\Gamma_{\Omega}} \Vert_{m+1}}{\Vert u-W_{p,m,\Gamma_{\Omega}} \Vert_{m+1}}=1+{\cal O}(H).
\end{equation}
\end{corollary}

\begin{proof}
From \eqref{eq:superconvergence}, Cauchy-Schwarz, and the triangle inequality
\begin{align}
\Vert u-U_{p,m,\Gamma_{\Omega}} \Vert_{m+1}&=\Vert u-W_{p,m,\Gamma_{\Omega}}+W_{p,m,\Gamma_{\Omega}}-U_{p,m,\Gamma_{\Omega}} \Vert_{m+1} \le \Vert u-W_{p,m,\Gamma_{\Omega}} \Vert_{m+1} \nonumber \\
&+C(p,m,P_{\rm max},p_{\rm max}^\prime,\vert \Omega \vert)H^{p-m+1}, \label{eq:equiva}
\end{align}
\noindent
and similarly
\begin{equation}
\Vert u-W_{p,m,\Gamma_{\Omega}} \Vert_{m+1} \le \Vert u-U_{p,m,\Gamma_{\Omega}} \Vert_{m+1}+C(p,m,P_{\rm max},p_{\rm max}^\prime,\vert \Omega \vert)H^{p-m+1}. \label{eq:equivb}
\end{equation}
\noindent
The result follows immediately from \eqref{eq:bound}-\eqref{eq:equivb}.
\end{proof}

\section{Computational Results}

The goal of this section is to examine the impact of $m$ on the other parameters $p$ and $N$.  While only one example is explored, the results are suggestive. To that end consider
\begin{equation}
0=-(-\Delta)^{m+1}u+u+f(x), \, x \in [0,1], \label{eq:test}
\end{equation}
\noindent
where the boundary conditions and $f(x)$ are chosen so that
\begin{equation}
u(x)={\rm tanh}10(x-1/2).
\end{equation}
\noindent
The solution profile mimics an interior boundary layer at $x=1/2$.  A finite element code written in Python using the Hermite-Jacobi basis has been developed.  It uses the package {\textit {mpmath}} \cite{mpmath} which allows for arbitrary precision arthimetic.  While the code therefore runs considerably slower, using {\textit {mpmath}} addresses the valid concerns about the loss of precision for high order methods in general. For example, in \cite{BS2018} the authors note that in the one-dimensional case the impact of round-off accumulation for fourth-order problems can be "seriously damaging" due to larger condition numbers than for second-order equations.   The {\textit {mpmath}} parameter that controls precision, $mp.dps$ is set to 120 (having run tests for both 80 and 160).

The relationship between the three parameters was explored through three tests.  The first two considered the behavior of the error in the interpolant
\begin{equation}
E_{I,m,p,N}\equiv\Vert u-W_{p,m,N} \Vert_{m+1}, \label{eq:interperror}
\end{equation}
\noindent
(where $\Gamma_{\Omega}$ has been replaced by the number of elements $N$) by examining its leading term.  From \eqref{eq:m1error} and Stirling's approximation it follows that
\begin{equation}
E_{I,m,p,N} \sim \frac{1}{2(2p-2m+1)}\left(\frac{he}{4(p-m)}\right)^{p-m}\left(\int_{\Omega}(u^{(p+1)}(x))^2dx \right)^{1/2}, \label{eq:testcurve}
\end{equation}
\noindent
where $p-m > e$ for $m \ge 2$ since $p \ge 2m+1$.  In the first experiment $p$ is left fixed so as to focus on the relationship between $N$ and $m$.  Thus, in \eqref{eq:testcurve}, $\int_{\Omega}(u^{(p+1)}(x))^2dx$ is held constant and for each $m \in \{2, 3, \ldots 8 \}$ $N$ is chosen to keep the left-hand side of \eqref{eq:testcurve} constant. In particular $p=17$ and the interpolation error is fixed at $1.5969 \times 10^{-4}$. To generate the curve specified by \eqref{eq:testcurve} \cite{mathematica} was used. The results shown in Figure \ref{figure1} indicate that to achieve the same error $N$ must increase exponentially as a function of $m$.  The finite element errors (label values) for the same $N$, $m$, and $p$ match the interpolant errors except when $m=2, 3$ when they are significantly smaller.  In these two cases, $N=6, 10$, the interpolant and finite element solution are in the pre-asymptotic range.

In the second experiment the left-hand side of \eqref{eq:testcurve} is fixed and for each $m$, $N$ is initially set to be 160 until $p$ is selected such that $E_{I,m,p,N} \le 1.2669 \times 10^{-10} < E_{I,m,p-1,N}$.  Then $N$ is decreased until $E_{I,m,p,N} \approx 1.2669 \times 10^{-10}$.  The number of degrees of freedom (dof) as a function of $m$ is displayed in Figure 2.  The growth in the dof is linear as is the growth of $p$ which satisfies $p=2m+6$.  The best fit slope for degrees of freedom for this example is equal to 225 which is not insubstantial. As in experiment one the finite element errors (label values) are close to the interpolant errors.

In experiment three the finite element error,
\begin{equation}
E_{m,p,N}=\Vert u-U_{p,m,N}\Vert_{m+1}    
\end{equation}
\noindent
is compared to $E_{m,p,2N}$ and to $E_{I,m,p,N}$.  To that end let
\begin{equation}
\Theta=\frac{E_{m,p,N}}{E_{I,m,p,N}}, \, \kappa=\frac{E_{m,p,N}}{E_{m,p,2N}},
\end{equation}
\noindent
Equation \eqref{eq:UFEMa} is solved on uniform grids with $\beta=1$ and thus $N=1/H$.  Table \ref{tab:table1} shows the results for $m=2, 5, 8$ and two values of $p$ for each $m$.  The smaller value of $p$ in each case, that is $p=5, 11, 17$, corresponds to a nodal basis only.  For all three values of $m$, $\Theta$ is 1 to five digits.  If interior basis functions are added $\Theta$ is slightly less than 1 until $N$ is sufficiently large.  The rate at which $\kappa$ approaches its theoretical limiting value decreases as $m$ increases.  The {\it a posteriori} error estimation strategy that guided the $h$-, $p$-, and $hp$-adaptivity strategies for $m=0$ and $m=1$ \cite{M2003} was more dependable when $\Theta$ and $\kappa$ were in the asymptotic range.   While obviously limited, these experiments suggest that as $m$ increases more computational resources will be required to achieve the same error and that the adaptive strategies will be less robust.  In particular the tests indicated that {\it h}-refinement on its own becomes a less effective strategy as $m$ increases.

\begin{table}[H]
\begin{center}
\begin{NiceTabular}{|c|c|c|c|c|c|}
\hline
 $p$ & $m$ & $N$ & $\|u-U_{p,m,\Gamma_{\Omega}}\|_{m+1}$ & $\Theta$ & $\kappa$ \\ 
\hline
5 & 2 & 10 & $4.5882 \times 10^{2}$ & 1.0000 & - \\
\hline
5 & 2 & 20 & $5.9066$ & 1.0000 & 7.7679 \\ 
\hline
5 & 2 & 40 & $6.9905 \times 10^{-1}$ & 1.0000 & 8.4495 \\ 
\hline
5 & 2 & 80 & $8.8280 \times 10^{-2}$ & 1.0000 & 7.9185 \\ 
\hline
5 & 2 & 160 & $1.1063 \times 10^{-2}$ & 1.0000 & 7.9795 \\ 
\hline
5 & 2 & 320 & $1.3838 \times 10^{-3}$ & 1.0000 & 7.9949 \\ 
\hline
\hline
10 & 2 & 10 & $2.90127 \times 10^{-1}$ & 0.99821 & - \\
\hline
10 & 2 & 20 & $5.7301 \times 10^{-5}$ & 0.99222 & 507.23 \\ 
\hline
10 & 2 & 40 & $4.99108 \times 10^{-7}$ & 0.99969 & 114.81 \\ 
\hline
10 & 2 & 80 & $1.9970 \times 10^{-9}$ & 0.99993 & 249.93 \\ 
\hline
10 & 2 & 160 & $7.83238 \times 10^{-12}$ & 0.99998 & 254.96 \\ 
\hline
10 & 2 & 320 & $3.0626 \times 10^{-14}$ & 1.0000 & 255.74 \\ 
\hline
\hline
11 & 5 & 10 & $9.7971 \times 10^4$ & 1.0000 & - \\
\hline
11 & 5 & 20 & $3.8238 \times 10^3$ & 1.0000 & 25.621 \\ 
\hline
11 & 5 & 40 & $5.4238 \times 10^{1}$ & 1.0000 & 70.502 \\ 
\hline
11 & 5 & 80 & $8.5042 \times 10^{-1}$ & 1.0000 & 63.778 \\ 
\hline
11 & 5 & 160 & $1.3367 \times 10^{-2}$ & 1.0000 & 63.622 \\ 
\hline
11 & 5 & 320 & $2.09167 \times 10^{-4}$ & 1.0000 & 63.905 \\ 
\hline
\hline
15 & 5 & 10 & $7.5180 \times 10^2$ & 0.9996 & - \\
\hline
15 & 5 & 20 & $7.3459 \times 10^{-1}$ & 0.99854 & 1023.4 \\ 
\hline
15 & 5 & 40 & $1.3382 \times 10^{-3}$ & 1.0000 & 548.94 \\ 
\hline
15 & 5 & 80 & $1.08232 \times 10^{-6}$ & 1.0000 & 1236.4 \\ 
\hline
15 & 5 & 160 & $1.06426 \times 10^{-9}$ & 1.0000 & 1017.0 \\ 
\hline
15 & 5 & 320 & $1.0411 \times 10^{-12}$ & 1.0000 & 1022.2 \\ 
\hline
\hline
17 & 8 & 10 & $2.3131 \times 10^9$ & 1.0000 & - \\
\hline
17 & 8 & 20 & $3.8548 \times 10^6$ & 1.0000 & 600.05 \\ 
\hline
17 & 8 & 40 & $2.8166 \times 10^4$ & 1.0000 & 136.86 \\ 
\hline
17 & 8 & 80 & $4.1533 \times 10^1$ & 1.0000 & 678.17 \\ 
\hline
17 & 8 & 160 & $8.1885 \times 10^{-2}$ & 1.0000 & 507.21 \\ 
\hline
17 & 8 & 320 & $1.60308 \times 10^{-4}$ & 1.0000 & 510.80 \\ 
\hline
\hline
22 & 8 & 20 & $3.07815 \times 10^{6}$ & 0.99987 & - \\
\hline
22 & 8 & 40 & $5.2073 \times 10^{2}$ & 0.9997 & 5,911.1 \\
\hline
22 & 8 & 40 & $2.2392 \times 10^{-2}$ & 0.9997 & 23,255.0 \\
\hline
22 & 8 & 80 & $2.0537 \times 10^{-6}$ & 0.9997 & 10,904.0 \\
\hline
22 & 8 & 160 & $1.2669 \times 10^{-10}$ & 1.0000 & 16,210.0 \\ 
\hline
22 & 8 & 320 & $7.7528 \times 10^{-15}$ & 1.0000 & 16,341.0 \\ 
\hline
\end{NiceTabular}
\end{center}
\caption{Finite element errors, error ratios, and convergence rates for a select set of $m$ and $p$ values.}
\label{tab:table1}
\end{table}

\section{Conclusions}
An explicit formula for Hermite splines for any interelement continuity $C^m$, $m \ge 0$, was presented.  These splines serve as the nodal basis functions for a finite element method for solving two-point boundary equations of arbitrary even order.  The splines are complemented by interior basis functions formed from ultraspherical polynomials multiplied by binomial powers of order $m+1$ at each element end point.  These polynomials were then used to construct an interpolant of the solution.  Formulas for the {\it a priori} error estimates of the interpolation error were presented and extensions of the Lagrange interpolants were introduced.  The leading term in the interpolation error was proportional to the $p+1${\it st} derivative of the solution multiplied by the next highest degree interior basis function. It was shown that the difference between the interpolant and the finite element solution in $H^{m+1}$ is one order higher than the {\it a priori} errors of both which in turn implies the asymptotic equivalence of the interpolant and the finite element solution. Computational results suggest that the performance gap between {\it hp}- and {\it h}-refinement grows with $m$.

The Hermite-Jacobi basis can be extended in a straightforward manner to tensor-product domains in two and three space dimensions (e.g., rectangles, rectangular prism).  They can also be used in adaptive solvers for higher-order time-dependent problems. The connection between the nodal basis functions and those in the interior suggest that this framework can be employed in various types of spectral and spectral element methods.  It is possible to show that the {\it a posteriori} error estimator developed for 2nd- \cite{M2001,M2003} and 4th-order \cite{MR2010} equations can be extended to any $m$ and $p$ which, in turn, can serve as the key driver in a {\it hp}-adaptive finite element solver.  Issues regarding precision as $m$ increases must also be resolved and adaptivity may be part of the solution.

\begin{figure}[H]
\caption{Plot of \eqref{eq:testcurve} (solid curve) as a function of $m$ and $N$ applied to \eqref{eq:test} with $p$ fixed so that the leading term in the interpolant error in $H^{m+1}(\Omega)$ is approximately $0.0001597$.  The labels  are corresponding errors in the finite element solution.}
\begin{center}
\includegraphics[width=15cm]{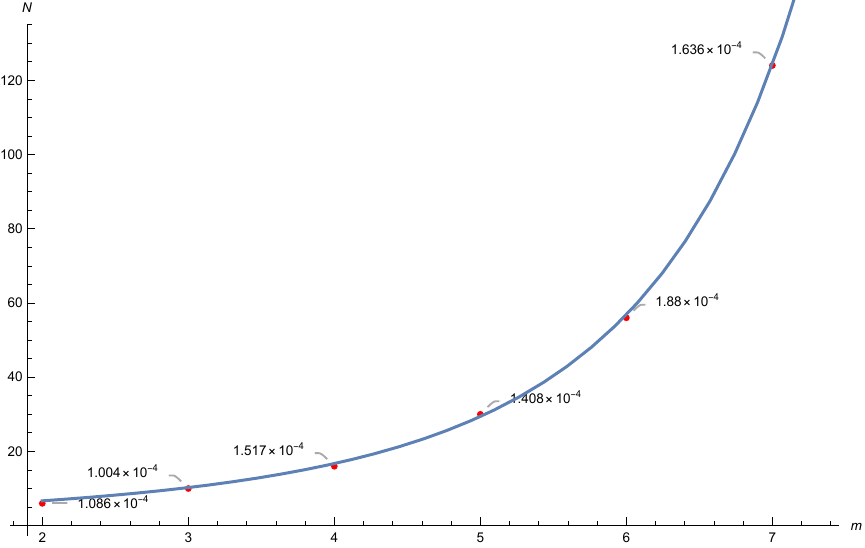}
\end{center}
\label{figure1}
\end{figure}

\begin{figure}[H]
\caption{Plot of the number of degrees of freedom as a function of $m$ so that the leading term in the interpolant error in $H^{m+1}(\Omega)$ is approximately $1.2675 \times 10^{-10}$.  The labels  are corresponding errors in the finite element solution.}
\begin{center}
\includegraphics[width=15cm]{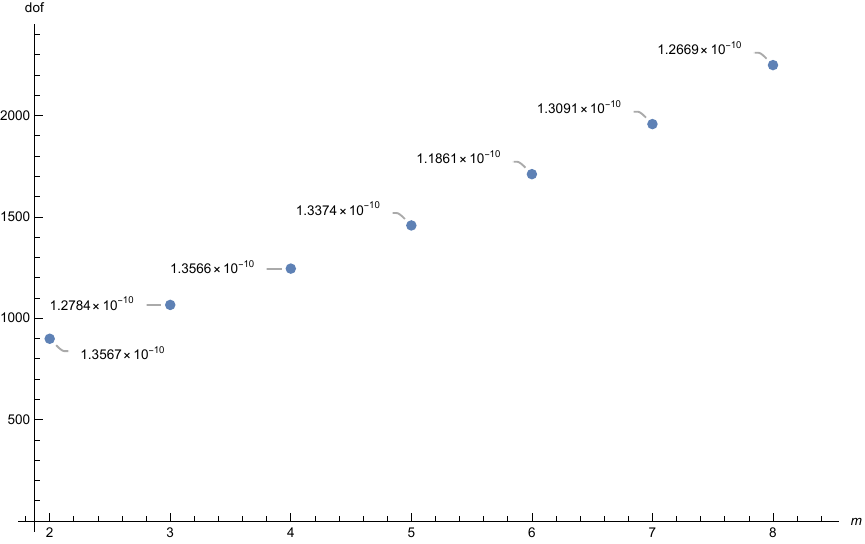}
\end{center}
\label{figure2}
\end{figure}

\section*{Acknowledgements}
The author thanks Tom Hagstrom for bringing to his attention the papers by Birkhoff, Schultz, Varga and by Spitzbart.

\end{document}